\documentclass[11pt]{amsart}

\usepackage{amsmath, amssymb, amsthm, amsfonts, amsxtra, amscd}

\input xy
\xyoption{all}

\usepackage{hyperref}

\swapnumbers
\numberwithin{equation}{section}

\theoremstyle{plain}
\newtheorem{theorem}[subsubsection]{Theorem}
\newtheorem{lemma}[subsubsection]{Lemma}
\newtheorem{prop}[subsubsection]{Proposition}
\newtheorem{cor}[subsubsection]{Corollary}

\theoremstyle{definition}
\newtheorem{defn}[subsubsection]{Definition}

\newtheorem{remark}[subsubsection]{Remark}
\newtheorem{exam}[subsubsection]{Example}
\newtheorem{ex}[subsubsection]{Exercise}

\def\AA{\mathbb{A}}

\def\CC{\mathbb{C}}

\def\FF{\mathbb{F}}
\def\GG{\mathbb{G}}

\def\PP{\mathbb{P}}
\def\QQ{\mathbb{Q}}

\def\TT{\mathbb{T}}

\def\ZZ{\mathbb{Z}}

\def\calE{\mathcal{E}}
\def\calF{\mathcal{F}}

\def\calI{\mathcal{I}}

\def\calK{\mathcal{K}}

\def\calM{\mathcal{M}}

\def\calO{\mathcal{O}}
\def\calP{\mathcal{P}}

\def\calV{\mathcal{V}}

\newcommand\cA{\mathcal{A}}

\newcommand\cE{\mathcal{E}}
\newcommand\cF{\mathcal{F}}

\newcommand\cH{\mathcal{H}}
\newcommand\cI{\mathcal{I}}

\newcommand\cK{\mathcal{K}}
\newcommand\cL{\mathcal{L}}

\newcommand\cO{\mathcal{O}}
\newcommand\cP{\mathcal{P}}

\newcommand\cV{\mathcal{V}}

\def\bI{\mathbf{I}}

\def\bK{\mathbf{K}}

\def\bP{\mathbf{P}}

\def\bR{\mathbf{R}}

\newcommand\frg{\mathfrak{g}}
\newcommand\frG{\mathfrak{G}}
\newcommand\tfrG{\widetilde{\mathfrak{G}}}

\newcommand\frk{\mathfrak{k}}
\renewcommand\frm{\mathfrak{m}}

\newcommand\tilA{\widetilde{A}}

\newcommand\tilL{\widetilde{L}}

\def\dG{\widehat{G}}
\def\dT{\widehat{T}}

\def\dB{\widehat{B}}

\newcommand\aff{\textup{aff}}

\newcommand\AS{\textup{AS}}

\newcommand{\Bun}{\textup{Bun}}
\newcommand{\chk}{\textup{char}(k)}

\newcommand\cont{\textup{cont}}

\newcommand{\Dyn}{\textup{Dyn}}
\newcommand\ev{\textup{ev}}

\newcommand\Frob{\textup{Frob}}

\newcommand\Gal{\textup{Gal}}

\newcommand\geom{\textup{geom}}

\newcommand{\Gr}{\textup{Gr}}

\newcommand\Hk{\textup{Hk}}
\newcommand\IC{\textup{IC}}
\newcommand\id{\textup{id}}

\newcommand{\ind}{\textup{ind}}

\newcommand\Kl{\textup{Kl}}
\newcommand\Lie{\textup{Lie}}
\newcommand\Loc{\textup{Loc}}

\newcommand\Perv{\textup{Perv}}
\newcommand{\Pic}{\textup{Pic}}
\newcommand\pr{\textup{pr}}

\newcommand\Rep{\textup{Rep}}

\newcommand\sep{\textup{sep}}
\newcommand\Span{\textup{Span}}
\newcommand\Spec{\textup{Spec}}
\newcommand\St{\textup{St}}

\newcommand\Stab{\textup{Stab}}

\newcommand\Sw{\textup{Sw}}

\newcommand{\Tr}{\textup{Tr}}

\newcommand{\Vect}{\textup{Vec}}

\newcommand\Aut{\textup{Aut}}
\newcommand\Hom{\textup{Hom}}

\newcommand\GL{\textup{GL}}

\newcommand\PGL{\textup{PGL}}

\newcommand\SL{\textup{SL}}

\newcommand\SO{\textup{SO}}
\newcommand\Spin{\textup{Spin}}

\newcommand\Sp{\textup{Sp}}

\newcommand{\Gm}{\GG_m}
\def\Ga{\GG_a}

\newcommand{\Ad}{\textup{Ad}}
\renewcommand\sc{\textup{sc}}

\newcommand\xch{\mathbb{X}^*}
\newcommand\xcoch{\mathbb{X}_*}

\newcommand{\incl}{\hookrightarrow}
\newcommand{\isom}{\stackrel{\sim}{\to}}

\newcommand{\surj}{\twoheadrightarrow}

\newcommand{\Qlbar}{\overline{\QQ}_\ell}

\newcommand{\hotimes}{\widehat{\otimes}}

\renewcommand{\j}[1]{\langle{#1}\rangle}
\newcommand{\wt}[1]{\widetilde{#1}}
\newcommand{\wh}[1]{\widehat{#1}}
\newcommand\quash[1]{}
\newcommand\mat[4]{\left(\begin{array}{cc} #1 & #2 \\ #3 & #4 \end{array}\right)}  

\newcommand{\bu}{\bullet}
\newcommand{\ov}{\overline}
\newcommand{\bs}{\backslash}
\newcommand{\tl}[1]{[\![#1]\!]}
\newcommand{\lr}[1]{(\!(#1)\!)}
\newcommand\sss{\subsubsection}
\newcommand\xr{\xrightarrow}
\newcommand\op{\oplus}
\newcommand\ot{\otimes}
\newcommand\bt{\boxtimes}
\newcommand\one{\mathbf{1}}

\renewcommand\c\circ
\newcommand\nth{^{\textup{th}}}

\newcommand{\cohog}[2]{\textup{H}^{#1}({#2})}     

\newcommand\upH{\textup{H}}

\newcommand{\oll}[1]{\overleftarrow{#1}}
\newcommand{\orr}[1]{\overrightarrow{#1}}

\renewcommand\a\alpha
\renewcommand\b\beta
\newcommand\g\gamma
\newcommand\G\Gamma
\renewcommand\d\delta
\newcommand\D\Delta
\newcommand{\e}{\epsilon}
\renewcommand{\th}{\theta}

\newcommand{\ph}{\varphi}
\renewcommand\r\rho
\newcommand{\s}{\sigma}
\renewcommand{\t}{\tau}

\newcommand{\y}{\eta}
\newcommand{\z}{\zeta}
\newcommand{\ep}{\epsilon}

\renewcommand{\l}{\lambda}
\renewcommand{\L}{\Lambda}
\newcommand{\om}{\omega}
\newcommand{\Om}{\Omega}
\newcommand{\Sig}{\Sigma}

\newcommand\na{\natural}

\newcommand{\diag}{\textup{diag}}
\newcommand{\Sat}{\textup{Sat}}

\newcommand\Gk{\Gal(\overline{k}/k)}

\newcommand\Qbar{\overline{\QQ}}

\newcommand{\kbar}{\overline{k}}

\newcommand{\xbar}{\overline{x}}

\newcommand\CS{\textup{CS}_{1}}

\newcommand\AG{G(F)\backslash G(\AA_{F})}
\renewcommand{\c}{\circ}

\newcommand\aud{automorphic datum}
\newcommand\gaud{geometric automorphic datum}
\newcommand\auda{automorphic data}

\newcommand\pt{\textup{pt}}
\newcommand\tilw{\widetilde{w}}
\newcommand\vn{\varnothing}

\title{Rigidity method for automorphic forms over function fields}
\author{Zhiwei Yun}
\thanks{Supported by the Simons Investigatorship and Packard Fellowship.}
\address{Department of Mathematics, Massachusetts Institute of Technology, 77 Massachusetts Ave, Cambridge, MA 02139}
\email{zyun@mit.edu}
\date{March 5, 2022}

\begin{document}



\maketitle

\setcounter{tocdepth}{1}

\tableofcontents

These are lectures given at the 2022 Arizona Winter School. It gives an introduction to the rigidity method for constructing automorphic forms for semisimple groups over function fields. The rigidity method leads to explicit constructions of local systems that are Langlands parameters of automorphic forms. The examples of local systems produced in this way have applications to algebraic geometry and number theory.

We will emphasize on the ``engineering'' aspect of the theory, leaving out most of the proofs:  we will give principles on how to design rigid \auda \  and methods for  computing the resulting local systems. For more details and proofs, we refer to \cite{Y-CDM}.

{\bf Acknowledgement}

I thank the organizers of the 2022 Arizona Winter School for inviting me to lecture,  and the University of Arizona and the Clay Institute of Mathematics for supporting my visit. I thank the participants, especially those in my project group, for attending my lectures and giving feedbacks. I am grateful to Benedict Gross and Konstantin Jakob for helpful comments.

\section{Automorphic forms over function fields}

In this section we recall some basic concepts for automorphic forms over a function field. 

\subsection{The setting} \label{ss:setting}

\subsubsection{Function field} Let $k$ be a finite field. Let $X$ be a projective, smooth and geometrically connected curve $X$ over $k$. Let $F=k(X)$ be the field of rational functions on $X$.

Let $|X|$ be the set of closed points of $X$ (places of $F$). For each $x\in |X|$, let $F_{x}$ denote the completion of $F$ at the place $x$. The valuation ring  and residue field of $F_{x}$ are denoted by $\calO_{x}$ and $k_{x}$. The maximal ideal of $\calO_{x}$ is denoted by $\frm_{x}$.  The ring of ad\`eles is the restricted product
\begin{equation*}
\AA_{F}:={\prod_{x\in |X|}}'F_{x}
\end{equation*}
where for almost all $x$, the $x$-component of an element $a\in \AA_{F}$ lies in $\calO_{x}$. There is a natural topology on $\AA_{F}$ making it a locally compact topological ring.

\subsubsection{Groups} For simplicity of the presentation, we restrict ourselves to the following situation: 
\begin{eqnarray*}
&&\mbox{\bf $G$ is a connected split reductive group over $k$;}\\
&&\mbox{\bf In addition, when talking about automorphic forms, $G$ is assumed to be semisimple.}
\end{eqnarray*}

We will fix a split maximal torus $T$ and a Borel subgroup $B\supset T$.  They give a based root system 
$$(\xch(T), \Phi(G,T), \D_{B}, \xcoch(T), \Phi^{\vee}(G,T), \D^{\vee}_{B})$$ 
where $\D_{B}$ are the simple roots and $\D_{B}^{\vee}$ the simple coroots. It makes sense to talk about $G(R)$ whenever $R$ is a $k$-algebra. For example, $G(F)$, $G(\calO_{x})$ and $G(F_{x})$, etc.

The Langlands dual group $\dG$ to $G$ is a connected reductive group over $\Qlbar$, equipped with a maximal torus $\dT$ and a Borel subgroup $\dB\supset \dT$, such that the corresponding based root system 
$$(\xch(\dT), \Phi(\dG,\dT), \Delta_{\dB}, \xcoch(\dT), \Phi^{\vee}(\dG,\dT), \D_{\dB}^{\vee})$$ 
is identified with the based root system $(\xcoch(T), \Phi^{\vee}(G,T), \Delta^{\vee}_{B}, \xch(T), \Phi(G,T), \D_{B})$ obtained from that of $G$ by switching roots and coroots.

\sss{Adelic and level groups}

The group $G(\AA_{F})$ of $\AA_{F}$-points of $G$ can also be expressed as the restricted product
\begin{equation*}
G(\AA_{F})={\prod_{x\in |X|}}'G(F_{x})
\end{equation*}
where most components lie in $G(\calO_{x})$.  This is a locally compact topological group. The diagonally embedded $G(F)$ inside $G(\AA_{F})$ is a discrete subgroup. The quotient $\AG$ is a locally compact space. 

Let $K_{x}\subset G(F_{x})$ be a compact open subgroup, one for each  $x\in |X|$, such that for almost all $x$, $K_{x}=G(\cO_{x})$. Let $K=\prod_{x} K_{x}$ be the compact open subgroup of $G(\AA_{F})$. The double coset space $\AG/K$ has the discrete topology.

\sss{Automorphic forms}
Let $\cA_{K}=C(\AG/K,\Qlbar)$ be the vector space of $\Qlbar$-valued functions on $G(\AA_{F})$ that are left invariant under $G(F)$ and right invariant under $K$. Let $\cA_{K,c}=C_{c}(\AG/K,\Qlbar)\subset \cA_{K}$ be the subspace that are supported on finitely many double cosets.

Elements in $\cA_{K,c}$ will be our working definition of (compactly supported) {\em automorphic forms for $G(\AA_{F})$ with level $K$}. For more precise definition, we refer to  \cite[Definition 5.8]{BorelJacquet}.

\sss{Action of the Hecke algebra}
Let $\cH_{K}$ be the vector space of $\Qlbar$-valued, compactly supported functions on $G(\AA_{F})$ that are left and right invariant under $K$. Similarly one can define the local Hecke algebra $\cH_{K_{x}}$. The vector space $\cH_{K}$ (resp. $\cH_{K_{x}}$) has an algebra structure under convolution, such that the characteristic function of $\one_{K}$ (resp. $\one_{K_{x}}$) is the unit element. Then $\cH_{K}$ is the restricted tensor product of the local Hecke algebras $\cH_{K_{x}}$ with respect to their algebra units. The algebra $\cH_{K}$ acts on $\cA_{K}$ and $\cA_{K,c}$ by the rule:
\begin{equation*}
(f\star h)(x)=\sum_{g\in G(\AA_{F})/K} f(xgK)h(Kg^{-1}K), \quad \forall f\in \cA_{K}, h\in \cH_{K}.
\end{equation*}

\begin{ex} Give a formula for the algebra structure of $\cH_{K}$.
\end{ex}

%

\subsubsection{Cusp forms} As our working definition, a  {\em cusp form} for $G(\AA_{F})$ with level $K$ is an element $f\in \cA_{K,c}$ that generates a finite-dimensional $\cH_{K}$-submodule. The official definition of a cusp form uses the vanishing of constant terms and looks quite different. It is shown in \cite[Proposition 8.23]{VLaff} that these two notions are equivalent. 

A cusp form $f\in \cA_{K,c}$ is called a {\em Hecke eigenform} if it is an eigenfunction for $\cH_{K_{x}}$, for almost all $x\in |X|$.


\subsection{Weil's interpretation}\label{ss:Weil} We allow $G$ to be  reductive in this subsection. Let $K^{\na}=\prod_{x}G(\cO_{x})$ in the following discussion.

\sss{Case of $\GL_{n}$}
When $G=\GL_{n}$,  Weil has given a geometric interpretation of the double coset $\AG/K^{\na}$ in terms of vector bundles of rank $n$ over $X$.  Let $\Vect_{n}(X)$ be the groupoid of rank $n$ vector bundles over $X$. Let us give the map $e: G(\AA_{F})\to \Vect_{n}(X)$.  First, for any finite subset $S\subset |X|$, we shall define a map $e_{S}: \prod_{x\in S}G(F_{x})\to \Vect_{n}(X)$ as follows. Let $j:X-S\incl X$ be the open  inclusion. Let  $(g_{x})_{x\in S}\in \prod_{x\in S}G(F_{x})$ and let $\L_{x}\subset F_{x}^{\op n}$ be the $\cO_{x}$-submodule $g_{x}\cO_{x}^{\op n}$. Define $e_{S}((g_{x})_{x\in S})$ to be the quasi-coherent subsheaf  $\cV\subset j_{*}\cO_{X-S}^{\op n}$ such that, for any affine open $U\subset X$,
\begin{equation*}
\G(U,\cV)=\G(U-S, \cO_{X}^{\op n})\cap (\prod_{x\in |U|\cap S}\L_{x}).
\end{equation*}
Here the intersection is taken inside $\prod_{x\in |U|\cap S}F^{\op n}_{x}$, and $\G(U- S, \cO_{X}^{\op n})$ is diagonally embedded into it by taking completion at each $x\in |U|\cap S$.

\begin{ex}
\begin{enumerate}
\item Show that $\cV$ constructed above is a vector bundle of rank $n$.
\item Suppose $S\subset S'$, $g_{x}\in G(F_{x})$ for each $x\in S$, and $g_{x}\in G(\cO_{x})$ for each $x\in S'- S$. Let $g_{S}=(g_{x})_{x\in S}$ and $g_{S'}=(g_{x})_{x\in S'}$. Then there is a canonical isomorphism $e_{S}(g_{S})=e_{S'}(g_{S'})$. Using this to show the $e_{S}$ for various $S$ give a well-defined map $e:G(\AA_{F})\to \Vect_{n}(X)$.
\item Show that $e$ is left invariant under $G(F)$ and right invariant under $K^{\na}$.
\item Show that $e$ is an equivalence of groupoids.
\end{enumerate}
\end{ex}

\sss{General $G$}
For general $G$, there is a similar interpretation of $\AG/K^{\na}$ in terms of $G$-torsors over $X$. Recall that a (right) $G$-torsor over $X$ is a scheme $Y\to X$ together with a fiberwise action of $G$ that looks like $G\times X$ (with $G$ acting on itself by right translation) \'etale locally over $X$. An isomorphism between $G$-torsors $Y$ and $Y'$ is a $G$-equivariant isomorphism $Y\cong Y'$ over $X$.

\begin{ex} When $G=\GL_{n}$, show that there is an equivalence of categories between $G$-torsors over $X$ and vector bundles of rank $n$ on $X$. You will need descent theory to show that \'etale local triviality is the same as Zariski local triviality for sheaves of $\cO_{X}$-modules.

Similarly, $\SL_{n}$-bundles are the same data as pairs $(\calV,\iota)$ where $\calV$ is a vector bundle of rank $n$ over $X$ and $\iota:\wedge^{n}\calV\isom\calO_{X}$ is a trivialization of the determinant of $\calV$. 
\end{ex}


\begin{ex} Show that $\PGL_{n}$-torsors over $X$ are the same as rank $n$ vector bundles over $X$ modulo tensoring with line bundles. For this,  you will need the fact that the Brauer group of $X$ is trivial.
\end{ex}

\begin{ex} For $G=\Sp_{2n}$ and $\SO_{n}$, give an interpretation of $G$-torsors using vector bundles with bilinear forms.
\end{ex}

Let $\Bun_{G}(k)$ be the groupoid of $G$-torsors over $X$: this is a category whose objects are $G$-torsors over $X$ and morphisms are isomorphisms between $G$-torsors. The groupoid $\Bun_{G}(k)$ is in fact the groupoid of $k$-points of an algebraic stack $\Bun_{G}$. For a $k$-algebra $R$ the groupoid $\Bun_{G}(R)$ is the groupoid of $G$-torsors over $X\otimes_{k}R$. 

Weil observed that there is a natural equivalence of groupoids
\begin{equation}\label{e}
e:\AG/K^{\na}\isom\Bun_{G}(k).
\end{equation}
So this is not just a bijection of sets,  but for any double coset $[g]=G(F)gK^{\na}$, the automorphism group of $e([g])$ (as a $G$-torsor) is isomorphic to $G(F)\cap gK^{\na}g^{-1}$.
 
The construction of $e$ is similar to the case of $G=\GL_{n}$, using modification of the trivial $G$-torsor at finitely many points $S$ given by $(g_{x})_{x\in S}\in \prod_{x\in S}G(F_{x})$.


\sss{Birkhoff decomposition} We consider the case $X=\PP^{1}$. 
Grothendieck proves that every vector bundle on $\PP^{1}$ is isomorphic to a direct sum of line bundles $\op_{i}\cO(n_{i})$, and the multiset $\{n_{i}\}$ is well-defined. This implies that the underlying set of $\Bun_{\GL_{n}}(k)$ is in bijection with the $\ZZ^{n}/S_{n}$. 

More generally, there is a canonical bijection of sets 
$$|\Bun_{G}(k)|\cong |\xcoch(T)/W|.$$

By the bijection \eqref{e}, we see that $|\AG/K^{\na}|$ is in bijection with $\xcoch(T)/W$.  We can construct the bijection as follows. 

\begin{ex} Let $t$ be an affine coordinate of $\AA^{1}\subset \PP^{1}$.
\begin{enumerate}
\item Let $G=T$ and $\l\in \xcoch(T)$. Viewing $\l$ as a homomorphism $\l:\Gm\to T$, and $t\in F_{0}$ (local field at $0\in |\PP^{1}|$), let $t^{\l}:=\l(t)\in T(F_{0})$. Now view $t^{\l}$  as an element in $T(\AA_{F})$ that is $t^{\l}$ at the place $0$ and $1$ elsewhere. The assignment $\l\to t^{\l}$ defines a homomorphism $\xcoch(T)\to T(\AA_{F})$. Show that it induces a bijection $\xcoch(T)\isom |T(F)\bs T(\AA_{F})/K^{\na}_{T}|$.

\item For general $G$, the above construction gives a map $\xcoch(T)\to |T(F)\bs T(\AA_{F})/K^{\na}_{T}| \to |\AG/K^{\na}|$. Show that this map is $W$-invariant,  and it induces a bijection between the set of $W$-orbits on $\xcoch(T)$ and the underlying set of $\AG/K^{\na}$.
\end{enumerate}

\end{ex}

\sss{Interpretation of Hecke operators as modifications}
By Weil's interpretation, $\cA_{K^{\na},c}$ can be identified with the space of $\Qlbar$-valued points on the set of isomorphisms classes of $G$-torsors over $X$, i.e.,
\begin{equation*}
\cA_{K^{\na}}=C(\Bun_{G}(k), \Qlbar).
\end{equation*}

Fix $x\in |X|$ and let $K_{x}=G(\cO_{x})$. The algebra $\cH_{K_{x}}$ is called the {\em spherical Hecke algebra} for the group $G(F_{x})$. It has a $\Qlbar$-basis consisting of characteristic functions $\one_{K_{x}gK_{x}}$ of double cosets.  We know that $\cH_{K_{x}}$ acts on $\cA_{K^{\na}}$. We give an interpretation of the action of $\one_{K_{x}gK_{x}}$ on $\cA_{K^{\na}}=C(\Bun_{G}(k), \Qlbar)$ in terms of $G$-torsors.

Take the example $G=\GL_{n}$ and 
$$g_{i,x}=\diag(\underbrace{t_{x},\cdots, t_{x}}_{i}, \underbrace{1,\cdots, 1}_{n-i})$$
($t_{x}$ is a uniformizer at $x$). For two vector bundles $\cE$ and $\cE'$ of rank $n$ on $X$, we use the notation $\cE\xr[x]{i}\cE'$ to mean that $\cE\subset \cE'\subset \cE(x)$ and  $\dim_{k_{x}}(\cE'/\cE)=i$. Then for any $f\in \cA_{K^{\na}}$, viewed as a function on $\Vect_{n}(X)\cong \Bun_{\GL_{n}}(k)$, we have
\begin{equation*}
(f\star\one_{K_{x}g_{i,x}K_{x}})(\cE)=\sum_{ \cE\xr[x]{i}\cE'} f(\cE').
\end{equation*}

\subsection{Level structures}

\sss{Parahoric subgroups} Let $x\in |X|$. Let $I\subset G(\cO_{x})$ be the preimage of a Borel subgroup $B(k_{x})\subset G(k_{x})$ under the reduction map $G(\cO_{x})\to G(k_{x})$. This is an example of an Iwahori subgroup of $G(F_{x})$. General Iwahori subgroups are conjugates of $I$ in $G(F_{x})$. Parahoric subgroups  of $G(F_{x})$ always contain an Iwahori subgroup with finite index.  A precise definition of parahoric subgroups involves a fair amount of Bruhat-Tits theory \cite{BT}. The conjugacy classes of parahoric subgroups under $G^{\sc}(F_{x})$ ($G^{\sc}$ is the simply-connected cover of $G$) are in bijection with
proper subsets of the vertices of the extended Dynkin diagram $\wt\Dyn(G)$ of $G$. Empty subset corresponds to Iwahori subgroups. The set of vertices of the finite Dynkin diagram corresponds to conjugates of $G(\cO_{x})$. 

Extended Dynkin diagram of $G$ is obtained from the Dynkin diagram of $G$ by adding one vertex corresponding to the affine simple root $\a_{0}$, which we mark as a black dot:
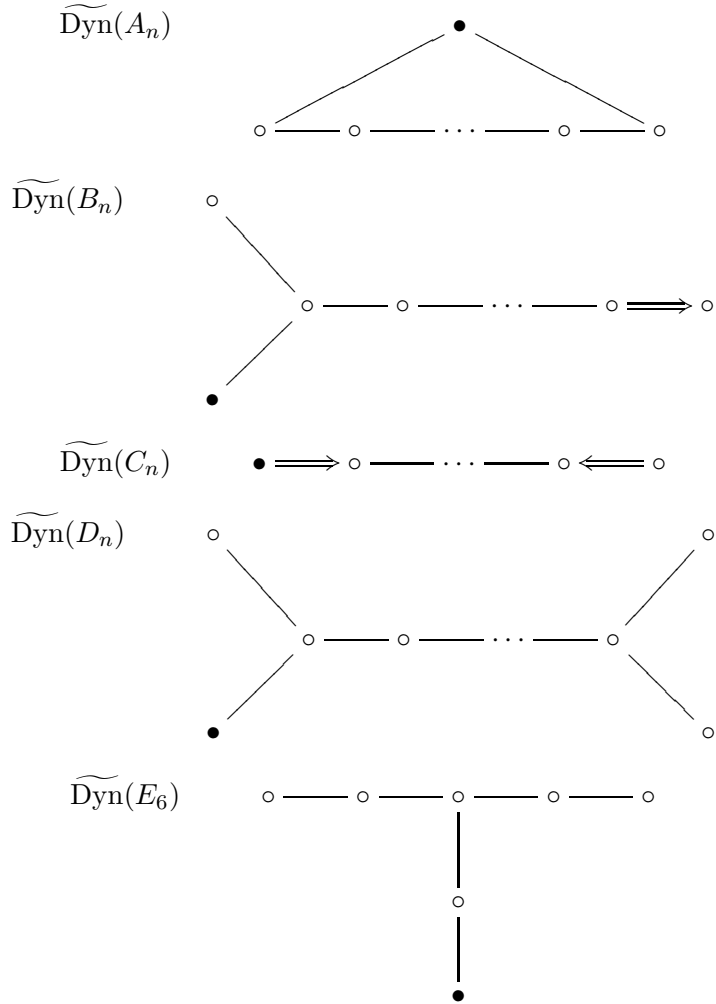
\begin{equation*}
\xymatrix{\wt \Dyn(A_{n}) & & & \bu\ar@{-}[dll]\ar@{-}[drr]\\
& \c\ar@{-}[r] & \c\ar@{-}[r] & \cdots \ar@{-}[r]  & \c\ar@{-}[r] &  \c
}
\end{equation*}
\begin{equation*}
\xymatrix{\wt \Dyn(B_{n}) & \c\ar@{-}[dr]\\
& & \c\ar@{-}[r] & \c\ar@{-}[r] & \cdots\ar@{-}[r] & \c\ar@2{->}[r] & \c\\
& \bu\ar@{-}[ur]
}
\end{equation*}
\begin{equation*}
\xymatrix{\wt \Dyn(C_{n})& \bu\ar@2{->}[r] &\c\ar@{-}[r] &\cdots\ar@{-}[r] &\c\ar@2{<-}[r] &\c
}
\end{equation*}
\begin{equation*}
\xymatrix{\wt \Dyn(D_{n})&\c\ar@{-}[dr] &&&&& \c\\
& & \c\ar@{-}[r] & \c\ar@{-}[r] & \cdots\ar@{-}[r] & \c\ar@{-}[ur]\ar@{-}[dr] \\
& \bu\ar@{-}[ur] &&&&& \c
}
\end{equation*}
\begin{equation*}
\xymatrix{
\wt \Dyn(E_{6})& \c\ar@{-}[r] &\c\ar@{-}[r] &\c\ar@{-}[r]\ar@{-}[d] &\c\ar@{-}[r] & \c\\
&&& \c\ar@{-}[d]\\
&&& \bu}
\end{equation*}
\begin{equation*}
\xymatrix{\wt \Dyn(E_{7})& \bu\ar@{-}[r] &\c\ar@{-}[r] &\c\ar@{-}[r] &\c\ar@{-}[r]\ar@{-}[d] &\c\ar@{-}[r] & \c\ar@{-}[r] &\c\\
&&&& \c
}
\end{equation*}
\begin{equation*}
\xymatrix{\wt \Dyn(E_{8})& \c\ar@{-}[r] &\c\ar@{-}[r] &\c\ar@{-}[r]\ar@{-}[d] &\c\ar@{-}[r] &\c\ar@{-}[r] &\c\ar@{-}[r] &\c\ar@{-}[r] &\bu\\
&&& \c
}
\end{equation*}
\begin{equation*}
\xymatrix{\wt \Dyn(F_{4})&\bu\ar@{-}[r] &\c\ar@{-}[r] &\c\ar@2{->}[r] &\c\ar@{-}[r] & \c
}
\end{equation*}
\begin{equation*}
\xymatrix{\wt \Dyn(G_{2})& \bu\ar@{-}[r] &\c\ar@3{->}[r] & \c
}
\end{equation*}

\begin{exam}\label{ex:SLn par} For $G = \SL_{n}$, we can understand parahoric subgroups using chains of lattices. A lattice in $F^{n}_{x}$ is an $\cO_{x}$-submodule of rank $n$. The relative dimension of a lattice $\L$ is defined to be 
\begin{equation*}
\dim[\L:\cO^{n}_{x}]=\ell_{x}(\L/\L\cap \cO^{n}_{x})-\ell_{x}(\cO^{n}_{x}/\L\cap\cO^{n}_{x})
\end{equation*}
where $\ell_{x}(-)$ is the length of a $\cO_{x}$-module (since we are in the function field case, $\ell_{x}(-)=\dim_{k_{x}}(-)$). A {\em periodic lattice chain} is a chain of lattices  $\L_{\bu}=\{\L_{i}\}_{i\in I}$ in $F^{n}_{x}$ indexed by a non-empty subset of $I\subset \ZZ$ such that
\begin{itemize} 
\item If $i<j$ then $\L_{i}\subset \L_{j}$;
\item The relative dimension of $\L_{i}$ is $i$.
\item $I$ is stable under translation by $n\ZZ$, and $\L_{i-n}=\frm_{x}\L_{i}$.
\end{itemize}
To a periodic lattice chain $\L_{\bu}$, its simultaneous stabilizer
\begin{equation*}
P_{\L_{\bu}}=\{g\in \SL_{n}(F_{x})|g\L_{i}=\L_{i}, \forall i\in I\}
\end{equation*}
is a parahoric subgroup of $\SL_{n}(F_{x})$. Conversely, every parahoric subgroup of $\SL_{n}(F_{x})$ arises this way for a unique periodic lattice chain $\L_{\bu}$.

A periodic  lattice chain is {\em complete} if $I=\ZZ$. 
Iwahori subgroups of $\SL_{n}(F_{x})$ are precisely the stabilizers of complete periodic  lattice chains.
\end{exam}

\begin{ex} Assume $\chk\ne2$.  Use lattice chains to describe parahoric subgroups of $G(F_{x})$ when $G=\Sp(V)$ or  $\SO(V)$, where $V$ is a vector space over $F_{x}$ equipped with a symplectic form or symmetric bilinear form $\j{\cdot,\cdot}$. For a lattice $\L\subset V$, denote $\L^{\vee}=\{v\in V|\j{v,\L}\subset \cO_{x}\}$, another lattice in $V$.
\begin{enumerate}
\item Let $G=\Sp(V)$. Let $\L\subset V$ be a lattice such that $\frm_{x}\L^{\vee}=\L$. Then its stabilizer $P_{\L}$ is a parahoric subgroup of $G(F_{x})$ whose conjugacy class corresponds  to the complement of the right end of $\wt\Dyn(G)$.  Show that $\L^{\vee}/\L$ as a $k_{x}$-vector space carries a canonical symplectic form, and give a canonical surjection $P_{\L}\to \Sp(\L^{\vee}/\L)$.

\item  The same construction ``works'' for $\SO(V)$ when $\dim V=2n$.   But to which subdiagram of $\wt\Dyn(G)$ does $P_{\L}$ correspond?

\item Let $G=\SO(V)$ and $\dim V=2n+1$. Let $\L\subset V$ be a lattice such that $\frm_{x}\L^{\vee}\subset \L$ with quotient of length $1$. Then its stabilizer $P_{\L}$ is a parahoric subgroup corresponding to the complement of the right end of $\wt\Dyn(G)$.  Show that $\L^{\vee}/\L$ as a $k_{x}$-vector space carries a canonical symmetric bilinear form, and give a canonical surjection $P_{\L}\to \SO(\L^{\vee}/\L)$.
\end{enumerate}
\end{ex}

\sss{Loop groups}

Later we shall also need to view $G(F_{x})$ and $G(\calO_{x})$ as infinite-dimensional groups over $k$. More precisely, for a $k_{x}$-algebra $R$ we let $L_{x}G(R)=G(R\hotimes_{k_{x}}F_{x})$ and let $L^{+}_{x}G(R)=G(R\hotimes_{k_{x}}\calO_{x})$. Here $R\hotimes_{k_{x}} F_{x}$ or $R\hotimes_{k_{x}} \calO_{x}$ means the completion of the tensor product with respect to the $\frm_{x}$-adic topology on $F_{x}$ or $\calO_{x}$.  If we choose a uniformizer $t_{x}$ of $F_{x}$, then $L_{x}G(R)=G(R\lr{t_{x}})$ and $L^{+}_{x}G=G(\tl{t_{x}})$. The functor $L_{x}G$ (resp. $L^{+}_{x}G$) is representable by a group indscheme (resp.  group scheme of infinite type) over $k_{x}$. There is a reduction map $L_{x}^{+}G\to G\otimes_{k}k_{x}$.

More generally, any parahoric subgroup of $G(F_{x})$ can be equipped with the structure of a pro-algebraic group (an inverse limit of algebraic groups).  Bruhat and Tits \cite{BT} constructed certain smooth models $\calP$ of $G$ over $\calO_{x}$  called Bruhat-Tits group schemes. The parahoric subgroups are exactly of the form $\cP(\cO_{x})$ for Bruhat-Tits group schemes $\cP$.  To such $\cP$ we can attach a pro-algebraic subgroup $\bP\subset L_{x}G$ whose $R$-points is $\calP(R\hotimes_{k_{x}}\calO_{x})$. Then $k_{x}$-points $\bP(k_{x})$ of $\bP$ is a parahoric subgroup of $G(F_{x})$. 

Viewed as a pro-algebraic group, the parahoric subgroup $\bP\subset L_{x}G$ has a maximal reductive quotient $L_{\bP}$ which is a connected reductive group over $k_{x}$. We denote $\bP^{+}=\ker(\bP\to L_{\bP})$ its pro-unipotent radical.

If the  conjugacy class  of $\bP$ corresponds to the subset $J$ of the vertices the extended Dynkin diagram $\wt\Dyn(G)$ of $G$, then the Dynkin diagram of $L_{\bP}$ is given by the subdiagram $\wt\Dyn(G)$ spanned by $J$.

For example,  the preimage $\bI_{x}$ of $B\otimes_{k}k_{x}$ under the reduction map $L_{x}^{+}G\to G\otimes_{k}k_{x}$ is called the {\em standard Iwahori subgroup} of $L_{x}G$ (with respect to $B$),  and any parahoric subgroup of $L_{x}G$ contains a conjugate of $\bI_{x}$.

\subsubsection{$\Bun_{G}$ with level structures}\label{sss:level str} One can generalize $\Bun_{G}$ to $G$-torsors with level structures. Fix a finite set $S\subset |X|$, and for each $x\in S$ let $\bK_{x}\subset L_{x}G$ be a proalgebraic subgroup commensurable with $L_{x}^{+}G$. Then we may talk about $G$-torsors over $X$ with $\bK_{S}$-level structures: these are $G$-torsors $\calE$ over $X$ together with trivializations $\iota_{x}:\calE|_{\Spec\calO_{x}}\cong G|_{\Spec\calO_{x}}$ (for each $x\in S$) up to left multiplication by $\bK_{x}$ (via  the intuitive action if $\bK_{x}\subset L_{x}^{+}G$, but it requires some thought to define the action in general). We shall denote the corresponding moduli stack by $\Bun_{G}(\bK_{S})$. 

Let $K_{x}=\bK_{x}(k_{x})\subset G(F_{x})$ be the corresponding compact open subgroup. Then the isomorphism \eqref{e} generalizes to an equivalence of groupoids
\begin{equation}\label{Weil Bun general}
G(F)\backslash G(\AA_{F})/(\prod_{x\notin S}G(\calO_{x})\times\prod_{x\in S}K_{x})\isom\Bun_{G}(\bK_{S})(k).
\end{equation}

\begin{ex} Let $G=\GL_{n}$ and let $\bK_{x}$ be an Iwahori subgroup of $L_{x}G$ for each $x\in S$. Give an interpretation of $\Bun_{G}(\bK_{S})(k)$ in terms of vector bundles on $X$ with extra structure along $S$.
\end{ex}

\section{Rigid automorphic data}

We will define the notion of an automorphic datum and its geometric version. We will upgrade the notion of automorphic functions to automorphic sheaves, and define what it means for such a \gaud\ to be rigid.

\subsection{Automorphic data}
Let $Z\subset G$ be the center of $G$. This is a finite group scheme over $k$.
\begin{defn} Let $S \subset |X|$ be finite.  An {\em automorphic datum} for $G$ unramified outside $S$ is a collection $\{(K_{x}, \chi_{x})\}_{x\in S}$ where
\begin{itemize}
\item $K_{x}$ is a compact open subgroup of $G(F_{x})$;
\item $\chi_{x}: K_{x}\to \Qlbar^{\times}$ is a finite order character.
\end{itemize}
We denote the automorphic datum by $(K_{S}, \chi_{S})$.
\end{defn}


\begin{defn}\label{def:cen} Let $\omega:Z(F)\backslash Z(\AA_{F})\to\Qlbar^{\times}$ be a central character.  and $(K_{S},\chi_{S})$ be an \aud. We say $\om$ is {\em compatible with } $(K_{S},\chi_{S})$ if 
$\om|_{Z(F_{x})}=1$ for $x\notin S$ and $\om|_{Z(F_{x})\cap K_{x}}=\chi_{x}|_{Z(F_{x})\cap K_{x}}$ for $x\in S$.
\end{defn}

\begin{defn} Let $(K_{S},\chi_{S})$ be an automorphic datum. A $\Qlbar$-valued function $f$ on $\AG$ is called {a \em $(K_{S},\chi_{S})$-typical automorphic form} if
\begin{enumerate}
\item For every $x\in S$ and $h_{x}\in K_{x}$,  $f(gh_{x})=\chi_{x}(h_{x})f(g)$ for all $g\in \AG$. 
\item For every $x\notin S$, $f$ is right $G(\cO_{x})$-invariant. 
\end{enumerate}
\end{defn}

Clearly, if $f$ is a nonzero $(K_{S}, \chi_{S})$-typical automorphic form that is eigen under the center $Z(\AA_{F})$ with central character $\omega$, then $\om$ has to be compatible with $(K_{S}, \chi_{S})$.

Let $\cA(K_{S},\chi_{S})$ be the $\Qlbar$-vector space of $(K_{S},\chi_{S})$-typical automorphic forms. Similarly we have the space $\cA_{c}(K_{S},\chi_{S})$ of compactly supported $(K_{S},\chi_{S})$-typical automorphic forms. 

Let $K^{+}_{x}=\ker(\chi_{x})$ for $x\in S$. If we define
\begin{equation*}
K^{+}=\prod_{x\in S}K^{+}_{x}\times \prod_{x\notin S}G(\cO_{x}),
\end{equation*}
then we have $\cA(K_{S},\chi_{S})\subset \cA_{K^{+}}$ and $\cA_{c}(K_{S},\chi_{S})\subset \cA_{K^{+},c}$.

\begin{exam}\label{ex:SL2 3pt} Let $X=\PP^{1}$,  $G=\SL_{2}$ and $S=\{0,1,\infty\}$. For $x\in S$, we let $K_{x}=I_{x}$ be the standard Iwahori subgroup of $G(F_{x})$, i.e.,
\begin{equation*}
I_{x}=\left\{\mat{a}{b}{c}{d}\in G(\cO_{x})\Big|c\in\frm_{x}\right\}.
\end{equation*}
For each $x\in S$, we choose a character $\chi_{x}:k^{\times}\to \Qlbar^{\times}$ and view it as a character of $I_{x}$ by sending the above matrix to $\chi_{x}(\overline{a})$, where $\overline{a}\in k^{\times}$ is the image of $a\in\calO_{x}^{\times}$. 

A central character compatible with $(K_{S}, \chi_{S})$ exists if and only if
\begin{equation}\label{det SL2}
\prod_{x\in S}\chi_{x}(-1)=1,
\end{equation}
in which case it is unique: take $\om_{x}=\chi_{x}|_{Z(F_{x})}$ for $x\in S$ and trivial otherwise. 
\end{exam}

\begin{exam}[Kloosterman \aud]\label{ex:Kl} In \cite{HNY}, an \aud\  for $X=\PP^{1}$ is introduced for any split reductive group and some quasi-split ones, which provided the first examples of rigid automorphic data. Below let $S=\{0,\infty\}$ and assume $G$ to be semisimple and split. 

At $0$ we take $K_{0}$ to be an Iwahori subgroup with quotient torus identified with $T$.  Let $\chi_{0}: T(k)\to \Qlbar^{\times}$ be a character, and view it as a character of $K_{0}$ via $K_{0}\to T(k)$. 

At $\infty$, let $\t=t^{-1}$, a uniformizer of $F_{\infty}\cong k\lr{\t}$. Let $I_{\infty}\subset G(\cO_{\infty})$ be the Iwahori subgroup that is the preimage of $B(k)$ under the reduction map $G(\cO_{\infty})\to G(k)$. Again we can have a natural map $I_{\infty}\to B(k)\to T(k)$; let $I^{+}_{\infty}$ be its kernel. We take $K_{\infty}=I^{+}_{\infty}$.

Let $\D_{B}=\{\a_{1},\cdots, \a_{r}\}$ be simple roots of $G$ with respect to $B$ and $T$. Let $U_{\a_{i}}\subset G$ be the root subgroups, each isomorphic to $\Ga$. Let $-\th\in \Phi(G,T)$ be the lowest root, and $U_{-\th}$ be the root subgroup. Let $U_{\a_{0}}$ be the additive group whose $k$-points are $U_{-\th}(\t \cO_{\infty})/U_{-\th}(\t^{2} \cO_{\infty})$.  There is a canonical surjective map
\begin{equation}\label{proj to roots}
\pr: K_{\infty}=I^{+}_{\infty}\surj\prod_{i=0}^{r}U_{\a_{i}}(k).
\end{equation}
We denote the kernel of this map by $I^{++}_{\infty}$.

To define $\chi_{\infty}$, pick an isomorphism $U_{\a_{i}}\cong \Ga$, and take the sum of these maps to give a homomorphism
\begin{equation*}
\ph: \prod_{i=0}^{r}U_{\a_{i}}(k)\to \Ga(k)=k.
\end{equation*}
Finally, compose with a nontrivial additive character $\psi$ of $k$ to get a character
\begin{equation*}
\chi_{\infty}: K_{\infty}\xr{\pr} \prod_{i=0}^{r}U_{\a_{i}}(k)\xr{\ph} k\xr{\psi} \Qlbar^{\times}.
\end{equation*}
For example, when $G=\SL_{2}$, we have
\begin{equation*}
I_{\infty}^{+}=\left\{\mat{a}{b}{c}{d}\Big|a,d\in 1+\t\cO_{\infty}, b\in \cO_{\infty}, c\in \t\cO_{\infty}\right\}
\end{equation*}
and, after identifying $U_{\a_{0}}$ and $U_{\a_{1}}$ with $\Ga$ in an obvious way, the map $\pr$ maps the above matrix to $(b\mod \t, c/\t \mod \t)\in k^{2}$. The \aud\ $(K_{S},\chi_{S})$ is called the {\em Kloosterman \aud}.

There is a unique central character $\om$ compatible with $(K_{S},\chi_{S})$:  $\om_{0}=\chi_{0}|_{Z(k)}$ and $\om_{\infty}=\chi_{0}^{-1}|_{Z(k)}$ (we identify $Z(F_{0})$ and $Z(F_{\infty})$ with $Z(k)$).   

The datum $(K_{\infty},\chi_{\infty})$ picks up those representations of $G(F_{\infty})$ that contain nonzero eigenvectors under $I^{+}_{\infty}$ on which $I^{+}_{\infty}$ acts through the character $\psi\circ\phi$. These representations were first discovered by Gross and Reeder \cite[\S9.3]{GR}, and are called {\em simple supercuspidal representations}. When $G$ is simply-connected, any such representation is given by compact induction $\ind^{G(F_{\infty})}_{Z(k)I^{+}_{\infty}}(\om_{\infty}\boxtimes\psi\circ\phi)$ for some  character $\om_{\infty}:Z(k)\to \Qlbar^{\times}$.
\end{exam}

\subsection{A naive definition of rigidity}
We would like to call an automorphic datum $(K_{S},\chi_{S})$ {\em rigid} if $\dim \cA_{c}(K_{S},\chi_{S})=1$, i.e., compactly supported $(K_{S},\chi_{S})$-typical automorphic forms are unique up to a scalar. For example, the automorphic datum in Example \ref{ex:SL2 3pt} turns out to be rigid in this sense. 

However, there are several issues with this definition:
\begin{itemize}
\item When $G$ is not simply-connected, $\Bun_{G}$ has several connected components, and it is more natural to require $(K_{S},\chi_{S})$-typical automorphic forms  to be unique up to scalar on each connected component. With this modification, the Kloosterman automorphic data in Example \ref{ex:Kl} are rigid.
\item It is more natural to fix a central character.
\end{itemize}
But these are only technical issues. The more serious issue is the following:
\begin{itemize}
\item The space $\cA_{c}(K_{S},\chi_{S})$ may be small simply because the field $k$ is small. It would be more natural to ask that the dimension of $\cA_{c}(K_{S},\chi_{S})$ be ``independent of $k$''.
\end{itemize}
What does it mean to vary $k$ while keeping the automorphic datum $(K_{S},\chi_{S})$? To make senses of it, we need to reformulate  automorphic datum in geometric terms.

\subsection{Sheaf-to-function correspondence}\label{ss:sheaf to fun}
\subsubsection{The dictionary}
Let $X$ be a scheme of finite type over a finite field $k$ and let $\calF$ be a constructible complex of $\Qlbar$-sheaves for the \'etale topology of $X$. For each closed point $x\in X$, the geometric Frobenius element $\Frob_{x}$ at $x$ acts on the geometric stalk $\calF_{\xbar}$, which is a complex of $\Qlbar$-vector spaces.  We consider the function
\begin{eqnarray*}
f_{\calF, k}: X(k)&\to& \Qlbar\\
x  &\mapsto& \sum_{i\in\ZZ}(-1)^{i}\Tr(\Frob_{x}, \upH^{i}\calF_{\xbar})
\end{eqnarray*}
Similarly we can define a function $f_{\calF,k'}:X(k')\to\Qlbar$ for any finite extension  $k'$ of $k$. The correspondence
\begin{equation*}
\calF\mapsto \{f_{\calF,k'}\}_{k'/k}
\end{equation*}
is called the {\em sheaf-to-function} correspondence. This construction enjoys various functorial properties. For a morphism $\phi: X\to Y$ over $k$, the derived push forward $f_{!}$ transforms into integration of functions along the fibers (this is a consequence of the Lefschetz trace formula for the Frobenius endomorphism); the derived pullback $f^{*}$ transforms into pullback of functions. It also transforms tensor product of sheaves into pointwise multiplication of functions.

\begin{defn}\label{def:char} Let $L$ be a connected algebraic group over $k$ with the multiplication map $m:L\times L\to L$ and the identity point $e:\Spec (k)\to L$. A {\em rank one  character sheaf} $\calK$ on $L$ is a local system of rank one on $L$ equipped with two isomorphisms
\begin{eqnarray*}
\mu: m^{*}\calK\isom\calK\boxtimes\calK,\\
u: \Qlbar\isom e^{*}\calK.
\end{eqnarray*}
These isomorphisms should be compatible in the sense that
\begin{eqnarray*}
\mu|_{L\times\{e\}}=\id_{\calK}\otimes u: \calK=\Qlbar\otimes_{\Qlbar}\calK\isom e^{*}\calK\otimes\calK,\\
\mu|_{\{e\}\times L}=u\otimes\id_{\calK}: \calK=\calK\otimes_{\Qlbar}\Qlbar\isom \calK\otimes e^{*}\calK.
\end{eqnarray*}
\end{defn}

Since $L$ is connected, the isomorphism $\mu$ in Definition \ref{def:char} automatically satisfies the usual cocycle relation on $L^{3}$.  A local system $\calK$ of rank one on $L$ being a character sheaf is a property rather than extra structure on $\calK$. Let $\CS(L)$ be the category (groupoid) of rank one character sheaves $(\calK,\mu,u)$ on $L$, then it carries a symmetric monoidal structure given by the tensor product of character sheaves with the unit object given by the constant sheaf. The groupoid $\CS(L)$ has trivial automorphisms, therefore it is equivalent to its underlying set. Tensor product equips $\CS(L)$ with the structure of an abelian group. 

If $f: L\to L'$ is a homomorphism between connected algebraic groups, then pullback along $f$ induces a group homomorphism $f^{*}: \CS(L')\to \CS(L)$.

One can define the group $\CS(L)$ of rank one character sheaves for $L$ a connected proalgebraic group. In fact, write $L$ as the inverse limit of finite-dimensional quotients $L_{i}$, and define $\CS(L)$ as the direct limit of $\CS(L_{i})$. 

We can similarly define the notion of rank one character sheaves over $\kbar$ and form the group $\CS(L/\kbar)$. The base change map $\CS(L)\to \CS(L/\kbar)$ is injective, and the image consists of $\Gk$-invariants.


For each $\calK\in\CS(L)$, the sheaf-to-function correspondence gives a function  $f_{\calK}:L(k)\to \Qlbar^{\times}$ which is in fact a group homomorphism because of the isomorphism $\mu$. This way we obtain a homomorphism
\begin{equation}\label{fL}
f_{L}:\CS(L)\to \Hom(L(k),\Qlbar^{\times}).
\end{equation}
One can show that $f_{L}$ is always injective. The following result gives descriptions of $\CS(L)$ in various special cases.

\begin{theorem}[{\cite[Appendix A]{Y-CDM}}]\label{th:char sh comm} 
\begin{enumerate}
\item Let $L$ be a connected commutative algebraic group over $k$. Then $f_{L}$ is an isomorphism of abelian groups
\begin{equation*}
f_{L}:\CS(L)\isom \Hom(L(k), \Qlbar^{\times}).
\end{equation*}
\item Let $L$ be a connected reductive group over $k$ and $L^{\sc}\to L$ be the simply-connected cover of its derived group. Then $f_{L}$ induces an isomorphism of abelian groups
\begin{equation*}
\CS(L)\isom \Hom(L(k)/L^{\sc}(k), \Qlbar^{\times}).
\end{equation*}
Let $T$ be a maximal torus in $L$ and $T^{\sc}\subset L^{\sc}$ be its preimage in $L^{\sc}$.  Then we also have
\begin{equation*}
\CS(L)\isom \Hom(T(k)/T^{\sc}(k), \Qlbar^{\times}).
\end{equation*}
\end{enumerate}
\end{theorem}
The construction of an inverse to $f_{L}$ uses the Lang map $L\to L$ sending $g\mapsto F_{L}(g)g^{-1}$, where $F_{L}$ is the Frobenius endomorphism of $L$ (relative to $k$).
  
\begin{exam}\label{ex:Kummer} 
When $L=\Gm$ is the multiplicative group,  local systems in $\CS(\Gm)$ are called {\em Kummer sheaves}.  They are in bijection with characters  $\chi:k^{\times}\to \Qlbar^{\times}$. Let $[q-1]: \Gm\to\Gm$ be the $(q-1)\nth$ power map, which is the Lang map for $\Gm$. Then the Kummer sheaf   associated with the character $\chi$ is $\cK_{\chi}:=([q-1]_{!}\Qlbar)_{\chi}$, the direct summand of $[q-1]_{!}\Qlbar$ on which $\Gm(k)=k^{\times}$ acts through $\chi$.

More generally, when $L$ is a torus, any object in $\CS(L)$ is obtained as a direct summand of $[n]_{!}\Qlbar$ where $[n]:T\to T$ is the $n$-th power morphism and $n$ is prime to $\chk$. If $L$ is a split torus, it suffices to take $n=q-1$. 
\end{exam}

\begin{exam}\label{ex:AS}
When $L=\Ga$ is the additive group,  local systems in $\CS(\Ga)$ are called {\em Artin-Schreier sheaves}.  They are in bijection with characters  $\psi:k\to \Qlbar^{\times}$.  Let $\l_{\Ga}:\Ga\to \Ga$ be the $\Ga(k)$-torsor given by $a\mapsto a^{q}-a$. Then the Artin-Schreier sheaves associated with $\psi$ is $\AS_{\psi}:=(\l_{\Ga,!}\Qlbar)_{\psi}$, the direct summand of $\l_{\Ga,!}\Qlbar$ on which $\Ga(k)=k$ acts through $\psi$.

More generally, let $L=V$ be a vector space over $k$ viewed as a vector group. Fix a nontrivial characters  $\psi:k\to \Qlbar^{\times}$. Then objects in $\CS(V)$ are of the form $\AS_{\phi}:=\phi^{*}\AS_{\psi}$ for a unique covector $\phi\in V^{*}$, viewed as a homomorphism $\phi:V\to \Ga$.
\end{exam}


\begin{ex} Let $L$ be a reductive group over $k$ such as $\GL_{n}$, $\PGL_{n}$ or $\SO_{n}$.  Describe elements in $\CS(L)$ using coverings of $L$.
\end{ex}

\subsection{Geometric automorphic data}\label{ss:geom data} We resume with the setting in \S\ref{ss:setting}. Also, for notational simplicity,  in the rest of the notes, we assume
\begin{equation}\label{Sk}
\mbox{$S\subset |X|$ is a finite set consisting of $k$-rational points.}
\end{equation}

\begin{defn}\label{def:geom data} A pair $(\bK_{S},\calK_{S})$ is a {\em geometric automorphic datum} with respect to $S$ if
\begin{enumerate}
\item $\bK_{S}$ is a collection $\{\bK_{x}\}_{x\in S}$, where  $\bK_{x}\subset L_{x}G$ is a connected proalgebraic subgroup contained in some parahoric subgroup.
\item $\calK_{S}$ is a collection $\{\calK_{x}\}_{x\in S}$ where $\calK_{x}\in\CS(\bK_{x})$. (i.e.,  $\calK_{x}$ is the pullback of a rank one character sheaf from a finite-dimensional quotient of $\bK_{x}$.)
\end{enumerate}
\end{defn}

\begin{remark}
\begin{enumerate}
\item It is possible to include a geometric central character into the definition of a geometric \aud, but we omit it here. For example, in case $\bK_{x}$ contains $Z$ for all $s\in S$, one datum to add is an trivialization of the character sheaf
\begin{equation*}
\bigotimes_{x\in S} \cK_{x}|_{Z} \in \CS(Z).
\end{equation*}
This is a geometric analogue of a compatibility condition in Definition \ref{def:cen},  but it is extra datum rather than a condition.

\item A geometric automorphic datum $(\bK_{S},\calK_{S})$ gives rise to an \aud\ $(K_{S},\chi_{S})$ where $K_{x}=\bK_{x}(k)$, and $\chi_{x}$ is the character of $K_{x}$ determined by $\cK_{x}$ under the sheaf-to-function correspondence (see \eqref{fL}). Since $k$ is a finite field, the character sheaf $\calK_{x}$ is uniquely determined by its associated character $\chi_{x}:K_{x}=\bK_{x}(k)\to\Qlbar^{\times}$. 

\item Typical examples of  $\bK_{x}$ are the {\em Moy-Prasad groups}. See \S\ref{sss:MP}.
\end{enumerate}
\end{remark}

\subsubsection{Base change of geometric automorphic data}\label{sss:bc} Let $k'/k$ be a finite extension. Let $S'$ be the preimage of $S$ in $X\otimes_{k}k'$ (by \eqref{Sk},  $S'$ also consists of $k'$-rational points).  Given a geometric automorphic datum $(\bK_{S}, \cK_{S})$, we may define a corresponding geometric automorphic datum $(\bK_{S'}, \chi_{S'})$ for $G$ and the function field $F'=F\otimes_{k}k'$. For each $y\in S'$ with image $x\in S$, let $\bK_{y}=\bK_{x}\otimes_{k}k'$, and let $\cK_{y}\in\CS(\bK_{y})$ be the  pullback of  $\calK_{x}$  along the projection $\bK_{y}\to \bK_{x}$. We have the space $\cA_{c}(k'; K_{S'},\chi_{S'})$ of compactly supported $(K_{S'},\chi_{S'})$-typical automorphic forms defined for the situation $X', G\ot_{k}k', K_{S'}$ and $\chi_{S'}$ (the last two come from $\bK_{S'}$ and $\cK_{S'}$).

\begin{defn} A  geometric automorphic datum $(\bK_{S}, \cK_{S})$ is called {\em weakly rigid}, if there is a constant $N$ such that for {\em every} finite extension $k'/k$,
\begin{equation*}
\dim_{\Qlbar} \cA_{c}(k'; K_{S'}, \chi_{S'})\le N,
\end{equation*}
and,  for some finite extension $k'/k$, $\cA_{c}(k'; K_{S'}, \chi_{S'})\ne0$.

\end{defn}

\subsection{Automorphic sheaves} 
Suggested by Weil's interpretation in \S\ref{ss:Weil} and the sheaf-to-function correspondence \S\ref{ss:sheaf to fun},  we shall seek to upgrade the function space $\cA_{c}(K_{S},\chi_{S})$ into a category of sheaves on the moduli stack of $G$-torsors over $X$ with level structures. This idea was due to Drinfeld who worked out the case $G=\GL_{2}$, and for general $G$ it was formulated as the geometric Langlands correspondence by Laumon \cite{Laumon} and Beilinson--Drinfeld \cite{BD}.

\subsubsection{The category of automorphic sheaves}\label{sss:auto sheaves}  Consider the situation in \S\ref{ss:geom data}. Let $\bK^{+}_{x}\subset\bK_{x}$ be a connected normal subgroup of finite codimension such that the rank one character sheaf $\calK_{x}$ on $\bK_{x}$ is pulled back from the quotient $L_{x}=\bK_{x}/\bK^{+}_{x}$. Let $\Bun_{G}(\bK_{S})$ and $\Bun_{G}(\bK^{+}_{S})$ be the moduli stack of $G$-torsors over $X$ with the respective level structures, as defined in \S\ref{sss:level str}.  The morphism $\Bun_{G}(\bK^{+}_{S})\to \Bun_{G}(\bK_{S})$ is an $L_{S}:=\prod_{x\in S}L_{x}$-torsor.  The tensor product $\calK_{S}:=\boxtimes_{x\in S}\calK_{x}$ is an object in $\CS(L_{S})$. It makes sense to talk about $\Qlbar$-complexes of sheaves on $\Bun_{G}(\bK^{+}_{S})$ which are $(L_{S},\calK_{S})$-equivariant.

Without giving the detailed definition, we have the full subcategory 
\begin{equation*}
D_{c}(\bK_{S},\cK_{S})\subset D^{b}_{(L_{S},\calK_{S})}(\Bun_{G}(\bK^{+}_{S}), \Qlbar)
\end{equation*}
consisting of bounded constructible $\Qlbar$-complexes on $\Bun_{G}(\bK^{+}_{S})$ equipped with $(L_{S},\calK_{S})$-equivariant structures that are supported on finite-type substacks of $\Bun_{G}(\bK^{+}_{S})$. Objects in the category $D_{c}(\bK_{S},\cK_{S})$ are called {\em $(\bK_{S},\cK_{S})$-typical  automorphic sheaves} with respect to the geometric automorphic datum $(\bK_{S},\cK_{S})$.

Let
\begin{eqnarray*}
K:=\prod_{x\notin S}G(\calO_{x})\times\prod_{x\in S}K_{x};\\
K^{+}:=\prod_{x\notin S}G(\calO_{x})\times\prod_{x\in S}K^{+}_{x}
\end{eqnarray*}
According to \eqref{Weil Bun general}, we have an $L_{S}(k)$-equivariant equivalence of groupoids 
$$\Bun_{G}(\bK^{+}_{S})(k)\cong \AG/K^{+},$$
that allows us to identify $(L_{S}(k), \chi_{S})$-eigenfunctions on both sides.  Therefore the sheaf-to-function correspondence  gives a map
\begin{equation*}
\textup{Ob }D_{c}(\bK_{S},\cK_{S})\to \cA_{c}(K_{S}, \chi_{S}). 
\end{equation*}
Similarly, for the base-changed situation to $k'/k$ including $k'=\ov k$, we get the category $D_{c}(k';\bK_{S'},\chi_{S'})$ for the base changed situation and a map
\begin{equation*}
\textup{Ob }D_{c}(k';\bK_{S'},\cK_{S'})\to \cA_{c}(k';K_{S'}, \chi_{S'}).
\end{equation*}

\subsubsection{Relevant points} Consider a point $\calE\in \Bun_{G}(\bK_{S})(\kbar)$, which represents  a $G$-torsor over $X_{\kbar}$ with $\bK_{x}$-level structures at $x\in S$. The automorphism group $\Aut(\calE)$ of the point $\calE$ is an affine algebraic group over $\kbar$. For each  $x\in S$, restricting an automorphism of $\calE$ gives an element in $\bK_{x}\otimes_{k}\kbar$ (this depends on the choice a trivialization of $\calE$ around $x$), and thus a homomorphism
\begin{equation}\label{evS}
\ev_{S,\calE}:\Aut(\calE)\to \prod_{x\in S} \bK_{x}\ot \ov k\to \prod_{x\in S}L_{x}\ot \ov k=:L_{S}\ot\kbar
\end{equation}
which is well-defined up to conjugacy if one changes the trivialization of $\calE$ around $x$.

\begin{defn}\label{def:relevant} Let $(\bK_{S},\cK_{S})$ be a geometric automorphic datum. A point $\calE\in \Bun_{G}(\bK_{S})(\kbar)$ is called {\em relevant to $(\bK_{S},\cK_{S})$} if the restriction of $\ev_{S,\calE}^{*}\calK_{S}$ to the neutral component $\Aut(\calE)^{\c}$ of $\Aut(\calE)$ is trivial as an element in $\CS(\Aut(\calE)^{\c})$. Otherwise the point $\calE$ is called {\em irrelevant}.
\end{defn}

We have a similar characterization of relevant points in terms of automorphic sheaves.
\begin{lemma}\label{l:irr sh} Let $\calE\in \Bun_{G}(\bK_{S})(\kbar)$. Then $\cE$ is relevant for the geometric \aud\ $(\bK_{S},\cK_{S})$ if and only if there exists $\calF\in D_{c}(\bK_{S},\cK_{S})$ such that the geometric stalk of $\cF$ at $\cE$ is nonzero. The same is true when ``stalk'' is replaced by ``costalk''.
\end{lemma}
\begin{proof}[Proof sketch] Let $A=\Aut(\cE)$ and $i_{\cE}: \pt/A\incl \Bun_{G}(\bK_{S})$ be the embedding. Let $\cL:=\ev^{*}_{S,\cE}\cK_{S}\in \CS(A)$. Then $i^{*}_{\cE}\cF$ is an $(A,\cL)$-equivariant sheaf on a point. If $\cL|_{A^{\c}}$ is nontrivial, such a sheaf must be zero. If $\cL|_{A^{\c}}$ is trivial, such a sheaf can be understood using twisted representations of $\pi_{0}(A)$, see \S\ref{ss:tw rep}. 
\end{proof}

\begin{defn} A geometric \aud\ $(\bK_{S}, \cK_{S})$ is called {\em geometrically rigid}, if for every connected component of $\Bun_{G}(\bK^{+}_{S})$, there is exactly one relevant $\ov k$-point.
\end{defn}

The following lemma gives another characterization of relevant points in terms of automorphic functions.
\begin{lemma}\label{l:irr} 
 Let $[g]\in\AG/K=\Bun_{G}(\bK_{S})(k)$,  viewed as a $\kbar$-point of $\Bun_{G}(\bK_{S})$. 
\begin{enumerate}
\item If $[g]$ is irrelevant, then any $f\in \cA_{c}(K_{S},\chi_{S})$ must vanish on the double coset $G(F)gK$. Similar statement holds when $k$ is replaced by a finite extension $k'$.
\item If $[g]$ is relevant, then for some finite extension $k'/k$, there exists a nonzero $f\in \cA_{c}(k'; K_{S'},\chi_{S'})$ supported on the double coset preimages of $G(F')gK'$ ($K'$ is defined as $K$ for the base changed situation).
\end{enumerate}
\end{lemma}

From this lemma,  we get the following result, which explains the relation between geometric rigidity and weakly rigidity.

\begin{theorem} Let $(\bK_{S},\cK_{S})$ be a geometric automorphic datum. 
\begin{enumerate}
\item The geometric \aud\ $(\bK_{S},\cK_{S})$ is weakly rigid if and only if there is a finite nonzero number of relevant $\kbar$-point on $\Bun_{G}(\bK_{S})$.
\item In particular, if $(\bK_{S},\cK_{S})$  is geomertically rigid, then it is weakly rigid.
\item If $(\bK_{S},\cK_{S})$ is weakly rigid, then for any finite extension $k'/k$, any $(K'_{S}, \chi_{S'})$-typical automorphic form $f$ is cuspidal. 
\end{enumerate}
\end{theorem}

\begin{exam} Consider the geometric \aud\ arising from Example \ref{ex:SL2 3pt}. The moduli stack $\Bun_{G}(\bI_{S})$ in this case classifies $(\calV,\iota,\{\ell_{x}\}_{x\in S})$ where $\calV$ is a rank two vector bundle over $X$, $\iota:\wedge^{2}(\calV)\cong\calO_{X}$ and $\ell_{x}$ is a line in the fiber $\calV_{x}$.

If $\chi_{S}$ satisfies the following genericity condition:
\begin{equation}\label{chi gen}
\mbox{for any map $\ep:S\to\{\pm1\} $, $\prod_{x\in S}\chi_{x}^{\ep(x)}$ is a nontrivial character on $k^{\times}$,}
\end{equation}
then $(\bI_{S}, \cK_{S})$ is geometrically rigid. This can be shown by doing the following Exercise. The only relevant point in $\Bun_{G}(\bI_{S})$ corresponds to the trivial bundle $\calO^{2}_{X}$ with three lines $\{\ell_{x}\}_{x\in S}$ in generic position (i.e., three distinct lines in $\ov k^{2}$).

\begin{ex} Consider the analogue of the \aud $(\bI_{S}, \chi_{S})$ as in Example \ref{ex:SL2 3pt} but for an arbitrary finite subset $S\subset \PP^{1}(k)$. Let $\cE$ be a rank two bundle over $\PP^{1}_{\ov k}$ with trivialized determinant and $\ell_{x}\in \cE_{x}$ be a line for each $x\in S$. We call $(\cE, \{\ell_{x}\}_{x\in S})$ {\em decomposable}, if $\cE$ can be written as a direct sum of line bundles $\cL'\op\cL''$ (over $\PP^{1}_{\ov k}$) such that for each $x\in S$, $\ell_{x}$ is either $\cL'_{x}$ or $\cL''_{x}$.
\begin{enumerate}
\item Show that if $(\cE, \{\ell_{x}\}_{x\in S})$ is decomposable, then the image of $\ev_{S,\cE}: \Aut(\cE, \{\ell_{x}\}_{x\in S})\to \GG_{m,\ov k}^{S}$ contains a subtorus of the form $\GG_{m,\ov k}\incl \GG_{m,\ov k}^{S}$ where each coordinate is either the identity map of $\Gm$ or the inversion map.
\item If $\chi_{S}$ satisfies \eqref{chi gen} and  $(\cE, \{\ell_{x}\}_{x\in S})$ is decomposable, then it is irrelevant for $(\bI_{S}, \chi_{S})$.
\item When $|S|\le 2$, show that all $(\cE, \{\ell_{x}\}_{x\in S})$ are decomposable.
\item When $|S|=3$, show that all $(\cE, \{\ell_{x}\}_{x\in S})$ are decomposable unless $\cE\cong \cO^{2}$.
\item Show that $(\cO^{2}, \{\ell_{x}\}_{x\in S})$ is indecomposable if and only if the set $\{\ell_{x}\}_{x\in S}$ contains at least $3$ distinct lines in $\ov k^{2}$. 
\item Consider the analogous situation for $G=\PGL_{2}$ and $|S|=3$. Which $(\cE, \{\ell_{x}\}_{x\in S})$ are indecomposable? 
\end{enumerate}

\end{ex}

\end{exam}

\begin{exam} Consider the geometric \aud\ arising from Example \ref{ex:Kl}. Let's assume $G$ is simply-connected so that $\Bun_{G}(\bK_{S})$ is connected. 

Let $B$ and $B^{op}$ be opposite Borel subgroups of $G$ containing $T$. Let $I_{0}\subset G(F_{0})$ be the Iwahori  that is the preimage of $B^{op}(k)\subset G(k)$; let $I_{\infty}^{+}\subset G(F_{\infty})$ be the Iwahori  that is the preimage of $B(k)\subset G(k)$.

Let $\G_{0}=I_{0}\cap G(\ov k[\t,\t^{-1}])$ ($\t=0$ corresponds to $t=\infty$), $\G^{+}_{\infty}=I_{\infty}^{+}\cap \G^{+}_{\infty}$. Then $\G_{0}\cap \G_{\infty}^{+}=\{1\}$. The $\ov k$-points of  $\Bun_{G}(\bK_{S})$ are in bijection with double cosets
\begin{equation}\label{Birk}
\G^{+}_{\infty} \bs G(\ov k[\t,\t^{-1}])/ \G_{0}.
\end{equation}
There is a bijection between the affine Weyl group $W_{\aff}=\xcoch(T)\rtimes W$ of $G$ and the double cosets \eqref{Birk}. To each $(\l,w)\in W_{\aff}$, the corresponding coset representative is $\t^{\l}\dot w$ for any lifting $\dot w$ to $N_{G}(T)$.

We show that the unit coset is the only relevant point for the Kloosterman \aud. Let $\tilw=(\l,w)\in W_{\aff}$ and let $\cE_{\tilw}\in \Bun_{G}(\bK_{S})(\ov k)$ be corresponding point. Then $\Aut(\cE_{w})=\G^{+}_{\infty}\cap \Ad(\tilw)\G_{0}$ is a finite-dimensional unipotent group.

To describe $\Aut(\cE_{w})$ we need some terminology from affine Lie algebras. The Lie algebra $\frg\ot \ov k[\t,\t^{-1}]$ of $G(\ov k[\t,\t^{-1}])$ decomposes as a direct sum of affine root spaces under the action of $T$ and the scaling on $\t$. For $\a\in \Phi(G,T)$, we say that $\frg_{\a}\t^{n}$ is the affine root space of the affine root $\a+n$. We say $\a+n>0$ if and only if $n>0$ or $(n=0)\wedge(\a>0)$. There is an action of $W_{\aff}$ on affine roots:  $(\l,w)$ sends $\a+n$ to $w\a+n+\j{\l,w\a}$. The extra affine simple root $\a_{0}=-\th+1$.

Then $\Aut(\cE_{w})$ is the unipotent group whose Lie algebra is spanned by the affine root spaces $\frg_{\a}\t^{n}$ such that
\begin{equation*}
\a+n>0, \tilw^{-1}(\a+n)<0.
\end{equation*}

When $\tilw\ne1$, we show $\cE_{\tilw}$ is irrelevant. one of the affine simple roots $\a_{i}$ will be sent to a negative affine root under $\tilw^{-1}$. Therefore $\Lie\Aut(\cE_{\tilw})$ contains $\a_{i}$, and the root subgroup $U_{\a_{i}}\subset \Aut(\cE_{\tilw})$. The evaluation map $\ev_{S}$ only maps nontrivially to $K_{\infty}$. Composing with $\ph$ to $\prod_{j}U_{\a_{j}}$, we see that
\begin{equation*}
\Aut(\cE_{\tilw})\xr{\ev_{\infty}}K_{\infty}=I^{+}_{\infty}\xr{\ph}\prod_{j}U_{\a_{j}}
\end{equation*}
 maps $U_{\a_{i}}$ identically to the $U_{\a_{i}}$-factor in the product. Therefore the Artin-Schreier sheaf $\AS_{\psi}$ pulls back nontrivially to $\Aut(\cE_{\tilw})$ (which is connected). This shows $\cE_{\tilw}$ is irrelevant.
\end{exam}

\begin{ex} For the Kloosterman \aud\ and  $G=\SL_{2}$, give an interpretation of $\Bun_{G}(\bK_{S})$ in terms of bundles with extra data. Classify points in this moduli stack, and show ``by hand'' that only the open point is relevant.
\end{ex}

\section{Constructing rigid \auda}

We will give several principles on how to construct  \auda\ that have a good chance of being weakly rigid. We also give a survey of most of the examples of rigid \auda\ in the literature.

\subsection{Numerical condition for rigidity}
All known examples of weakly rigid \gaud\ $(\bK_{S}, \cK_{S})$ satisfy
\begin{equation*}
\dim\Bun_{G}(\bK_{S})=0.
\end{equation*}
When $(\bK_{S}, \cK_{S})$ is geometrically rigid and its relevant points have finite stabilizer, the above equality automatically holds.

Suppose $X=\PP^{1}$. We give a formula for $\dim\Bun_{G}(\bK_{S})$.  Let $\frk_{x}\subset \frg(F_{x})$ be the Lie algebra of $\bK_{x}$. We have the relative dimension $[\frg(\cO_{x}): \frk_{x}]$ over $k_{x}$ (see Example \ref{ex:SLn par}).

\begin{lemma} Let $g_{X}$ be the genus of $X$. We have
\begin{equation*}
\dim\Bun_{G}(\bK_{S})=\sum_{x\in S}[\frg(\cO_{x}): \frk_{x}]+(g_{X}-1)\dim G
\end{equation*}
\end{lemma}

Therefore when $X=\PP^{1}$, $\dim\Bun_{G}(\bK_{S})=0$ is equivalent to
\begin{equation}\label{dim condition}
\sum_{x\in S}[\frg(\cO_{x}): \frk_{x}]=\dim G.
\end{equation}
This is a numerical condition that the \aud\ wants to satisfy.

\begin{exam} Suppose $X=\PP^{1}$ and all $\bK_{S}$ are parahoric subgroups with Levi quotient $L_{x}$. Assume $\bK_{x}$ contains the standard Iwahori $\bI_{x}$. Then 
$$[\frg(\cO_{x}): \frk_{x}]=\dim L^{+}_{x}G/\bI_{x}- \dim \bK_{x}/\bI_{x}=\dfrac{\dim G-\dim L_{x}}{2}=\dfrac{|\Phi( G,T)|-|\Phi (L_{x},T)|}{2}.$$ 

One can read of the root system of $L_{x}$ from the extended Dynkin diagram $\wt\Dyn(G)$, so the condition \eqref{dim condition} is easy to verify in this case.

Here are some interesting examples in exceptional groups when $S=\{0,1,\infty\}$ and $\bK_{x}$ are parahoric subgroups. In each case,  we describe the subset of the extended Dynkin diagram correponding to the parahoric subgroups $\bK_{0}, \bK_{1}$ and $\bK_{\infty}$.

\begin{enumerate}
\item Tetrahedron Example.  Let $G=G_{2}$:

\begin{equation*}
\xymatrix{ \stackrel{\a_{0}}\c\ar@{-}[r] &\stackrel{\a_{1}}\c\ar@3{->}[r] & \stackrel{\a_{2}}\c}
\end{equation*}
\begin{itemize}
\item $\bK_{0}$ corresponds to $\{\a_{0}, \a_{2}\}$ so $L_{0}\cong \SO_{4}$;
\item  $\bK_{1}$ and $\bK_{\infty}$ correspond to  single vertices (not necessarily the same one) so $L_{1}$ and $L_{\infty}$ are isogenous to $\GL_{2}$.
\end{itemize}

\item Octahedron Example. Let $G=F_{4}$:
\begin{equation*}
\xymatrix{\stackrel{\a_{0}}\c\ar@{-}[r] &\stackrel{\a_{1}}\c\ar@{-}[r] &\stackrel{\a_{2}}\c\ar@2{->}[r] &\stackrel{\a_{3}}\c\ar@{-}[r] & \stackrel{\a_{4}}\c
}
\end{equation*}
\begin{itemize}
\item $\bK_{0}$ corresponds to $\{\a_{0}, \a_{2}, \a_{3}, \a_{4}\}$ so that $L_{0}$ is isogenous to $\SL_{2}\times \Sp_{6}$;
\item  $\bK_{\infty}$ corresponds to  $\{\a_{0}, \a_{1}, \a_{3}, \a_{4}\}$ so that $L_{\infty}$ is isogenous to $\SL_{3}\times \SL_{3}$;
\item $\bK_{1}$ corresponds to $\{\a_{1}, \a_{3}, \a_{4}\}$ so that $L_{1}$ is isogenous to $\GL_{2}\times \SL_{3}$.
\end{itemize}

\item Icosahedron Example. Let $G=E_{8}$:

\begin{equation*}
\xymatrix{ \stackrel{\a_{1}}\c\ar@{-}[r] &\stackrel{\a_{3}}\c\ar@{-}[r] &\stackrel{\a_{4}}\c\ar@{-}[r]\ar@{-}[d] &\stackrel{\a_{5}}\c\ar@{-}[r] &\stackrel{\a_{6}}\c\ar@{-}[r] &\stackrel{\a_{7}}\c\ar@{-}[r] &\stackrel{\a_{8}}\c\ar@{-}[r] &\stackrel{\a_{0}}\c\\
&&\stackrel{\a_{2}}\c& 
}
\end{equation*}

\begin{itemize}
\item $\bK_{0}$ corresponds to the complement of $\a_{1}$ so that $L_{0}$ is isogenous to $\Spin_{16}$;
\item  $\bK_{\infty}$ corresponds to the complement of  $\a_{2}$ so that $L_{\infty}$ is isogenous to $\SL_{9}$;
\item $\bK_{1}$ corresponds to the complement of  $\a_{5}$ so that $L_{1}$ is isogenous to $\SL_{5}\times \SL_{5}$.
\end{itemize}
\end{enumerate}
\end{exam}

How are these examples related to Platonic solids? It turns out that in each example above, $\dim L_{x}=|\Phi|/n_{x}$ for some positive integer $n_{x}$ ($x\in S$). The triple $(n_{0},n_{\infty}, n_{1})$ in the three cases are:
\begin{enumerate}
\item $G=G_{2}$: $(n_{0},n_{\infty}, n_{1})=(2,3,3)$; 
\item $G=F_{4}$: $(n_{0},n_{\infty}, n_{1})=(2,3,4)$; 
\item $G=E_{8}$: $(n_{0},n_{\infty}, n_{1})=(2,3,5)$.
\end{enumerate}
  
\begin{ex} Construct more examples with all $\bK_{S}$ are parahoric subgroups satisfying \eqref{dim condition}.
\end{ex}

\subsection{Designing rigid \auda}
The choice of local data $(\bK_{x},\cK_{x})$ is guided under the local Langlands correspondence by the structure theory of $p$-adic groups on the one hand, and local Galois representations on the other hand.

\sss{Moy Prasad filtration}\label{sss:MP}
For a parahoric subgroup $\bP$ in the loop group $LG=G\lr{t}$, Moy and Prasad defined a natural filtration on it indexed by $\frac{1}{m}\ZZ_{\ge0}$: 
\begin{equation*}
\bP=\bP_{0}\supset \bP_{\ge 1/m}\supset \bP_{\ge 2/m}\supset\cdots.
\end{equation*}
Here the number $m$ can be calculated as follows. Let $\{n_{i}\}_{0\le i\le r}$ be the usual labelling on the extended Dynkin diagram:
\begin{equation*}
n_{0}=1, \quad \th=\sum_{i=1}^{r}n_{i}\a_{i}.
\end{equation*}
Let $J\subset \{0,1,\cdots, r\}$ corresponds to the subset of vertices in $\wt\Dyn(G)$ corresponding to $\bP$. Then
\begin{equation*}
m=\sum_{i\in J}n_{i}.
\end{equation*}

Let us assume $\bP$ contains the standard Iwahori subgroup $\bI$.  Let $x_{\bP}\in \xcoch(T)_{\QQ}$ be determined by the condition
\begin{equation*}
\j{\a_{i},x}=\begin{cases} 0, &  i\in J \\ 1/m, & i\notin J\end{cases} \quad 1\le i\le r.
\end{equation*}
This is the barycenter of the facet in the building of $LG$ corresponding to $\bP$. Then the Moy-Prasad filtration on $\bP$ can be described as follows:
\begin{enumerate}
\item $\bP_{\ge 1/m}=\bP^{+}$ is the pro-unipotent radical of $\bP$.
\item The quotient $V_{i/m}=\bP_{\ge i/m}/\bP_{\ge (i+1)/m}$ is a vector group over $k_{x}$. The affine roots that appear in $V_{i/m}$ are those  $\a+n$ (where $\a\in \Phi(G,T)\cup\{0\}$, $n\in\ZZ$) such that
\begin{equation*}
\j{\a, x_{\bP}}+n=i/m.
\end{equation*}
Here we also allow $\a$ to be zero, in which case the affine root space for $n$ is $t^{n}\Lie(T)$.
\end{enumerate}

\begin{ex} Describe the Moy-Prasad filtration for the standard Iwahori subgroup.
\end{ex}

\subsubsection{Local monodromy}\label{sss:loc mon} Let $x\in S$.  Under the local Langlands correspondence, an irreducible representation $\pi_{x}$ of $G(F_{x})$ should give rise to a homomorphism $\r_{x}: \Gal(F^{\sep}_{x}/F_{x})\to \dG(\Qlbar)$. We will discuss the global Langlands correspondence in \S\ref{ss:global Langlands}. If an automorphic representation $\pi$ of $G(\AA_{F})$ gives rise to a homomorphism $\r: \Gal(F^{\sep}/F)\to \dG(\Qlbar)$, then  $\r_{x}$ should be the restriction of $\r$ to the decomposition group $\Gal(F^{\sep}_{x}/F_{x})$ at $x$. If $\r$ comes from a $\dG$-local system $E$ on $U=X-S$, then $\r_{x}$ records the $\dG$-local system on $\Spec F_{x}$ obtained by restricting $E$ to the formal punctured disk at $x$.

The local Galois group $\Gal(F^{\sep}_{x}/F_{x})$ is an extension of the form
\begin{equation*}
1\to \calI_{x}\to \Gal(F^{\sep}_{x}/F_{x})\to\Gal(\kbar/k_{x})\to 1. 
\end{equation*}
The normal subgroup $\calI_{x}$ of $\Gal(F^{\sep}_{x}/F_{x})$ is  the {\em inertia group} at $x$. We have a normal subgroup $\calI^{w}_{x}\lhd\calI_{x}$ called the {\em wild inertia group} such that the quotient $\calI^{t}_{x}:=\calI_{x}/\calI^{w}_{x}$ is the maximal prime-to-$p$ quotient of $\calI_{x}$, called the {\em tame inertia group}. We have a canonical isomorphism of $\Gal(\kbar/k_{x})$-modules
\begin{equation*}
\calI^{t}_{x}\isom\varprojlim_{(n,p)=1}\mu_{n}(\kbar).
\end{equation*}
The $\dG$-local system $\rho$ is said to be {\em tame at $x\in S$} if $\rho_{x}|_{\calI_{x}}$ factors through the tame inertia group $\calI^{t}_{x}$. 

For representation $\s$ of $\Gal(F^{\sep}_{x}/F_{x})$ on an $n$-dimensional $\Qlbar$-vector space $V$, there is a multiset of $n$ non-negative rational numbers called the {\em slopes} of $\s$. The slopes measure how wildly ramified $\r_{x}$ is. Tame representations have all slopes equal to $0$. The sum of all slopes (with multiplicities) is a non-negative  integer called the Swan conductor $\Sw(\s)$. The Artin conductor of $\s$ is defined as
\begin{equation*}
a(\s)=\dim V/V^{\s(\cI_{x})}+\Sw(\s).
\end{equation*}
Of particular importance to us is when $\s=\Ad(\r_{x})$, which is the composition $\r_{x}: \Gal(F^{\sep}_{x}/F_{x})\to \dG(\Qlbar)\to\GL(\wh\frg)$.

\sss{Matching \auda \ with local monodromy}\label{sss:matching loc mono}

The local Langlands correspondence (which is still a conjecture beyond $\GL_{n}$) allows one to predict the inertial part of the local Galois representation from the datum $(\bK_{x},\cK_{x})$; conversely, if we want to design an \aud\ to match a given  representation $\r_{x}|_{\cI_{x}}: \cI_{x}\to \dG(\Qlbar)$, the local Langlands correspondence tells us which $(\bK_{x}, \cK_{x})$ to choose.

We give some principles of this sort.
\begin{enumerate}
\item If $(\bK_{S},\cK_{S})$ is weakly rigid, and it corresponds to a local $\dG$-Galois representation $\r_{x}$ at $x\in S$,  then we expect the matching of local numerical invariants:
\begin{equation*}
[\frg(\cO_{x}): \frk_{x}]\stackrel{?}{=}a(\Ad(\r_{x})).
\end{equation*}
For intuition behind this expectation, see \S\ref{sss:coh rig}.

\item If $\bK_{x}$ is a parahoric subgroup and $\cK_{x}$ is pulled back from the  Levi quotient $L_{x}$, then $\r_{x}$ should be tame. Moreover, one can determine the semisimple part of $\r_{x}|\cI_{x}$ as follows.  We may assume $T\subset L_{x}$. Restrict $\cK_{x}\in \CS(L_{x})$ to $T$ we get a Kummer local system which is in bijection with characters $\chi: T(k_{x})\to\Qlbar^{\times}$. Equivalently, $\chi$ can be viewed as a homomorphism $k_{x}^{\times}\ot_{\ZZ}\xcoch(T)\to\Qlbar^{\times}$, hence giving  $\chi': k_{x}^{\times}\to \Hom(\xch(T), \Qlbar^{\times})=\dT(\Qlbar)$. By local class field theory we have $\cI_{x}^{ab}\cong \cO_{x}^{\times}$ hence a canonical tame quotient $\cI_{x}\surj k_{x}^{\times}$. Then $\chi'$ can be viewed as a homomorphism
\begin{equation*}
\wt\chi': \cI_{x}\surj k_{x}^{\times}\xr{\chi'} \dT(\Qlbar).
\end{equation*}
Then $\wt\chi'$ is conjugate to the semisimple part of $\r_{x}|_{\cI_{x}}$ .

How about the unipotent part of $\r_{x}|_{\cI_{x}}$? When $(\bK_{S}, \cK_{S})$ is weakly rigid, the unipotent part of $\r_{x}|_{\cI_{x}}$ tends to be as nontrivial as possible, given that it commutes with its semisimple part which is determined by $\cK_{x}$ as described above.

\item If all slopes of $\r_{x}$ are $\le \l$ for some $\l\in \QQ$, then there should exist a parahoric  subgroup $\bP\subset L_{x}G$ such that $\bK_{x}$ contains $\bP_{\ge \l}$. Moreover, one can take $\bK^{+}_{x}$ to contain $\bP_{>\l}$.

For example, when $G=\PGL_{n}$, we will see in \S\ref{ss:Kl} that the Galois representation arising from the Kloosterman \aud\  is the classical Kloosterman sheaf of Deligne. Its slopes at $\infty$ are all equal to $1/n$. On the other hand, $\bK_{\infty}=\bI^{+}_{\infty}$ is the $1/n$ step in the Moy-Prasad filtration of $\bI_{\infty}$.  The more epipelagic examples \ref{ex:ep} also follow this pattern.

\item More generally, we expect $\bK_{x}$ to be a {\em Yu group}, following a construction of J-K. Yu \cite{Yu}. Roughly, if $\l_{0}<\l_{1}<\l_{2}<\cdots$ are the different slopes that appear in the local adjoint Galois representation $\Ad (\r_{x})$, one should try to find a sequence of twisted Levi subgroups
\begin{equation*}
G^{0}\subset G^{1}\subset G^{2}\subset \cdots \subset G
\end{equation*}
and Moy-Prasad groups $\bP^{i}_{\ge\l_{i}/2}\subset L_{x}G^{i}$ (again this is oversimplified!), and take  $\bK_{x}$ to be the subgroup generated by the $\bP^{i}_{\ge\l_{i}/2}$.  There is also a recipe for matching the wild part of $\cK_{x}$ (which corresponds to characters of $\bP^{i}_{\ge\l_{i}}/\bP^{i}_{>\l_{i}}$) with the restriction of $\r_{x}$ to the wild inertia.
\end{enumerate}

\subsection{More examples}

\begin{exam}[Hypergeometric \auda] In \cite{KY}, Kamgarpour and Yi constructed \auda\ realizing Katz's hypergeometric local systems.  Let $G=\PGL_{n}$, $X=\PP^{1}$. There are two kinds of such data corresponding to tame and wild hypergeometric sheaves.

The tame case generalizes Example \ref{ex:SL2 3pt}. In this case, $S=\{0,1,\infty\}$. Both $\bK_{0}$ and $\bK_{\infty}$ are taken to be an Iwahori subgroup. Let $\cK_{0}$ and $\cK_{\infty}$ be arbitrary Kummer sheaves on the quotient tori of $\bK_{0}$ and $\bK_{\infty}$. Take $\bK_{1}\subset L^{+}_{1}G$ to be parahoric subgroup that is the preimage of a parabolic subgroup with block sizes $(n-1,1)$ under the evaluation map $L^{+}_{1}G\to G$ (there are two conjugacy classes of such parabolic subgroups, but their preimages in $L^{+}_{1}G$ are conjugate under $\PGL_{n}(F_{1})$). Let $\cK_{1}$ be any Kummer sheaf pulled back from the abelianization map $\bK_{1}\to \Gm$.

In the wild case, take $S=\{0,\infty\}$. Let $\bK_{0}$ be an Iwahori subgroup and $\cK_{0}$  be a Kummer sheaf pulled back from its quotient torus. 	To define $(\bK_{\infty}, \cK_{\infty})$, we recall that the hypergeometric sheaf of rank $n$ we are trying to get under the Langlands correspondence will have $d$ slopes equal to $1/d$ (for some $1\le d\le n$), and $n-d$ slopes equal to $0$. Therefore it is natural to ask $\bK_{\infty}$ to contain $\bP^{+}=\bP_{1/d}$ for some parahoric subgroup with the first nontrivial filtration step $1/d$, together with a bit more to accommodate a Kummer sheaf on $\Gm^{n-d}$ for the tame part. Let $V=k^{n}$ and think of $\PGL_{n}$ as $\PGL(V)$. Choose a partial flag
$$V_{\bu}: \quad 0=V_{0}\subset V_{\le 1}\subset V_{\le 2}\subset\cdots V_{\le d}=V $$
with $V_{i}=V_{\le i}/V_{\le i-1}$ nonzero. Fix a decomposition $V_{i}=\ell_{i}\op V'_{i}$ ($1\le i\le d$). Let $\bP\subset L_{\infty}G$ be the preimage of the parabolic $P_{V_{\bu}}\subset \PGL(V)$ stabilizing the above partial flag. Then its Levi quotient $L_{\bP}\cong (\prod_{i=1}^{d}\GL(V_{i}))/\D\Gm$. Let $\bP^{++}=\bP_{2/d}$. Then the next subquotient in the Moy-Prasad filtration has the form
\begin{equation*}
\bP^{+}/\bP^{++}\cong \bigoplus_{i=1}^{d}\Hom(V_{i},V_{i-1})
\end{equation*}
as an $L_{\bP}$-module. Here $V_{0}$ is understood to be $V_{d}$. 

Let $\ph: \bP^{+}/\bP^{++}\to \Ga$ to be a linear function whose component $\ph_{i}: V_{i}\to V_{i-1}$ sends the line $\ell_{i}$ isomorphically to $\ell_{i-1}$, and is zero on $V'_{i}$. Let $L_{\ph}$ be the stabilizer of $\ph$ under $L_{\bP}$; then $L_{\ph}\cong \prod_{i}\GL(V'_{i})$. Let $B_{\ph}\subset L_{\ph}$ be a Borel subgroup. Then its quotient torus $T_{\ph}$ has dimension $n-d$. Finally, let $\bK_{\infty}$ be the preimage of $B_{\ph}$ under the quotient $\bP\to L_{\bP}$. 

To define $\cK_{\infty}$, we first extend $\ph^{*}\AS_{\psi}\in \CS(\bP^{+})$ to a rank one character sheaf $\cK_{\ph}$ on $\bK_{\infty}$ (which is possible because $B_{\ph}$ fixes $\ph$). Let $\cK_{\r}$ be the pullback of any Kummer sheaf on $T_{n-d}$ via $\bK_{\infty}\to B_{\ph}\to T_{\ph}$. Finally take $\cK_{\infty}=\cK_{\ph}\ot\cK_{\r}\in\CS(\bK_{\infty})$. 

The construction of $(\bK_{\infty},\cK_{\infty})$ in the wild case can be generalized to other types of groups, see Example \ref{ex:eu}.

For a specific choice of the partial flag $V_{\bu}$ where the dimensions of $V_{i}$ are distributed as evenly as possible, it was proved in \cite{KY} that $(\bK_{S},\cK_{S})$ is geometrically rigid.
\end{exam}

\begin{exam}[Airy \aud] In the work of Kamgarpour, Jakob and Yi \cite{JKY}, a series of geometrically rigid examples with $X=\PP^{1}$ and $S=\{\infty\}$ are introduced. When $G=\SL_{2}$ the corresponding local system, or rather its de Rham version, is the $D$-module given by the Airy equation. 

We explain the construction for  $G=\SL_{2}$. Let $K_{\infty}$ be the subgroup of $I_{\infty}$ consisting of the following matrices (where $a_{i}, b_{i}, c_{i}\in k$)
\begin{equation*}
\mat{1+a_{1}\t+\cdot }{b_{0}+b_{1}\t+\cdot }{b_{0}\t+c_{2}\t^{2}+\cdots}{1-a_{1}\t+\cdots}
\end{equation*}
There is a homomorphism $\ph: K_{\infty}\to k$ sending the above matrix to $b_{1}+c_{2}$ (check this is a group homomorphism!). Let $\chi_{\infty}$ be the pullback of a nontrivial additive character via $\ph$. It is clear how to turn $(K_{\infty}, \chi_{\infty})$ into a geometric \aud \ $(\bK_{\infty}, \cK_{\infty})$.

A more conceptual way of defining $K_{\infty}$ is the following. First take the Moy-Prasad filtration $I_{3/2}$ of $I_{\infty}$ and consider the homomorphism $\ph_{3/2}: I_{3/2}\to I_{3/2}/I_{2}\cong k^{2}\xr{sum} k$. We may ask if $\ph_{3/2}$ can be extended to a larger subgroup of $I^{+}=I_{1/2}$. First $\ph_{3/2}$ can be extended to $\ph_{1}: I_{1}\to k$ by letting it to be trivial on the diagonal matrices.  Next, we may view $\ph_{1}$ as an element (not uniquely determined) in the dual of the Lie algebra $\frg\lr{\t}$, which can be identified with $\frg\lr{\t}\frac{d\t}{\t}$ using the Killing form and the residue pairing. We choose $\wt\ph_{1}=\mat{0}{\t^{-2}}{\t^{-1}}{0}\frac{d\t}{\t}$. Let $H(F_{\infty})$ be the centralizer of $\wt\ph_{1}$ under the adjoint action of $G(F_{\infty})$. We see that
\begin{equation*}
H(F_{\infty})=\left\{\mat{a}{b}{b\t}{a}\Big|a^{2}-b^{2}\t=1, a,b\in F_{\infty}\right\}.
\end{equation*}
In fact $b\in \cO_{\infty}$ and $a\in 1+\t\cO_{\infty}$. Finally $K_{\infty}=H(F_{\infty})I_{1}$.  
\end{exam}

\begin{exam}[Epipelagic \auda]\label{ex:ep}
In \cite{Y-Epi} we give generalizations of Kloosterman sheaves. Let $X=\PP^{1}$ and $S=\{0,1\}$. Let $\bP_{\infty}\subset L_{\infty}G$ be a parahoric subgroup of specific types.  These parahoric subgroups are singled out by Reeder and Yu \cite{RY} to define their  {\em epipelagic representations}. Their conjugacy classes are in bijection with regular elliptic conjugacy classes in the Weyl group $W$. Consider the next step $\bP^{++}_{\infty}=\bP_{\infty, 2/m}$ in the Moy-Prasad filtration of $\bP_{\infty}$, and $V_{\bP}:=\bP^{+}_{\infty}/\bP^{++}_{\infty}$ is a finite-dimensional representation of the Levi $L_{\bP}$. 

We take $\bK_{\infty}=\bP^{+}_{\infty}$. To define $\cK_{\infty}$, we pick a linear function $\ph: V_{\bP}\to k$ satisfying the stability condition: it has closed orbit and finite stabilizer under $L_{\bP_{\infty}}$. Define $\cK_{\infty}$ to be the pullback of $\AS_{\psi}$ under the composition
\begin{equation*}
\bK_{\infty}=\bP^{+}_{\infty}\surj V_{\bP}\xr{\ph} \Ga.
\end{equation*}

For $\bK_{0}$, we take it to be the parahoric subgroup $\bP_{0}\subset L_{0}G$ opposite to $\bP_{\infty}$: the intersection $\bP_{0}\cap \bP_{\infty}$ is a common Levi subgroup $L_{\bP}$ of both. We can take any  $\cK_{0}\in \CS(L_{\bP})$. The geometric \aud\ $(\bK_{S}, \cK_{S})$ is geometrically rigid. To check this, one has to use the stability property of $\ph$.

One can let $\phi$ varies in an open subset of the dual space of $V_{\bP}$ and get a family version of geometric \aud. The resulting $\dG$-local systems ``glue'' together to give a $\dG$-local system over an open subset of $V^{*}_{\bP}\times(\PP^{1}-\{0,\infty\})$.
\end{exam}

\begin{exam}[Euphotic \auda, {\cite{JY}}]\label{ex:eu} This is a further generalization of the epipelagic \auda\ that weakens the assumption that $\ph\in V_{\bP}^{*}$ is stable. Let $X=\PP^{1}$ and $S=\{0,\infty\}$. Assume $G$ is simply-connected.

We start with any parahoric $\bP_{\infty}$ and any linear function $\ph$ on $V_{\bP}=\bP^{+}_{\infty}/\bP_{\infty}^{++}$ with a closed orbit under $L_{\bP}$. Let $L_{\ph}$ be the stabilizer of $\ph$ in $L_{\bP}$, and let $B_{\ph}\subset L_{\ph}$ be a Borel subgroup. Then $\bK_{\infty}=\bP^{+}_{\infty}B_{\ph}$. We take $\cK_{\infty}$ to be the tensor product of $\ph^{*}\AS_{\psi}$ (extended to a character sheaf on $\bK_{\infty}$ by letting it be trivial on $B_{\ph}$) and a Kummer sheaf $\cK_{\chi}$ from the quotient torus of $B_{\ph}$.

At $0$ we take $\bK_{0}$ to be a parahoric subgroup contained in the parahoric $\bP_{0}$ opposite to $\bP_{\infty}$. It corresponds to the choice of a parabolic subgroup $Q\subset L_{\bP}$. We require that
\begin{equation}\label{Q open orbit}
\mbox{The group $B_{\ph}$ acts on  $L_{\bP}/Q$ with an open orbit with finite stabilizer.}
\end{equation}
This is to match the numerical condition $\dim\Bun_{G}(\bK_{S})=0$. Let $\cK_{0}$ be trivial.

Such \aud\ $(\bK_{S}, \cK_{S})$ is not guaranteed to be weakly rigid. In \cite{JY},  for $P_{\infty}=G(\cO_{\infty})$, we give a complete list of pairs $(\ph, Q)$  (where $\ph\in\frg^{*}$ is semisimple, and $Q\subset G$ is a parabolic subgroup) satisfying the condition \eqref{Q open orbit}. In \cite{JY} we prove that in these cases, for a generic choice of the Kummer sheaf $\cK_{\chi}$ on $B_{\ph}$, $(\bK_{S}, \cK_{S})$ is weakly rigid, but not always geometrically rigid.
\end{exam}

\begin{exam} We resume the setup of Example \ref{ex:eu}. We give a series of examples of  euphotic \auda\ that are weakly rigid but not geometrically rigid. 

Consider the case $G=\SO(V)$ with $\dim_{k} V=2n+1$ and $\bP_{\infty}=G(\cO_{\infty})$. Write $V=I\op k\cdot e_{0}\op I^{*}$ where $I$ is a maximal isotropic subspace of $V$ and $(e_{0},e_{0})=1$. Let $\ph\in V_{\bP_{\infty}}^{*}\cong \frg$ be the element $\id_{I}\op 0\op (-\id_{I^{*}})$. Then the centralizer $G_{\ph}=L_{\ph}\cong \GL(I)$. Choose a Borel subgroup $B_{\ph}\subset G_{\ph}=\GL(I)$ and denote it by $B_{I}$, so that $\bK_{\infty}=G(\cO_{\infty})^{+}B_{I}$. Let $\bK_{0}\subset G(\cO_{0})$ be the preimage of the parabolic subgroup $Q_{I}\subset G$, the stabilizer of $I$. Then $\Bun_{G}(\bK_{\infty}, \bK_{0})$ contains an open substack isomorphic to $B_{I}\bs G/Q_{I}$ (this is the locus where the underlying $G$-bundle is trivial). For $\cE$ in this open substack, the local system $\ph^{*}\AS_{\psi}$ is always trivial  on the image of $\ev_{\infty,\cE}:\Aut(\cE)\to \bK_{\infty}$. So $\cE$ is relevant if and only if $\cK_{\chi}$ restricts to the trivial local system on the image of $\ev_{\infty,\cE}$. In particular, $\cE$ is relevant if $\Aut(\cE)^{\c}$ is  unipotent. If $\cE$ corresponds to a point $gQ_{I}\in G/Q_{I}$, then $\Aut(\cE)$ is the stabilizer of $gQ_{I}$ under the left translation by $B_{I}$. So we should look for points on $G/Q_{I}$ with unipotent stabilizers under $B_{I}$.

Assume $n=2m$ is even in the following discussion.
The partial flag variety $G/Q_{I}$ classifies $n$-dimensional isotropic subspaces of $V$. Consider the ``big cell'' $Y\subset G/Q_{I}$ consisting of isotropic $J\subset V$ such that $J$ projects isomorphically to $I$. Then $Y\cong I^{*}\op \wedge^{2}(I^{*})$. There is an open  $B_{I}$-orbit $O\subset \wedge^{2}(I^{*})$. We will exhibit $2^{m}$ $B_{I}$-orbits on $I^{*}\times O$ with unipotent stabilizers, and hence all of them are relevant for the euphotic \aud\ $(\bK_{S}, \chi_{S})$. 

The Borel subgroup $B_{I}$ is the stabilizer of a flag $0\subset I_{1}\subset\cdots\subset I_{n-1}\subset I_{n}=I$. A point $\om\in O\subset\wedge^{2}(I^{*})$ is a symplectic form on $I$ that is non-degenerate on $I_{2i}$ for each $i=1,\cdots, m$. Fix such an $\om$, and we get an orthogonal decomposition $I=N_{1}\op N_{2}\op\cdots \op N_{m}$  under $\om$ with $\dim N_{i}=2$ such that $I_{2i}=N_{1}\op\cdots\op N_{i}$. Each $N_{i}$ has a line $\ell_{i}\subset N_{i}$ that is the image of $I_{2i-1}$. We have $\Stab_{B_{I}}(\om)\cong \prod_{i=1}^{m}\SL(N_{i})$.  The $B_{I}$-orbits on $I^{*}\times O$ are in bijection with $\prod_{i=1}^{m}\SL(N_{i})$-orbits on $I^{*}=\op_{i=1}^{m}N_{i}^{*}$. Let $\e: \{1,\cdots, m\}\to \{0,1\}$ be any map. Let $I^{*}_{\ep}\subset I^{*}$ be the set of vectors $(\xi_{i})_{1\le i\le m}$ with $\xi_{i}\in N^{*}_{i}-\{0\}$, such that $\xi_{i}(\ell_{i})=0$ if $\e(i)=0$ or $\xi_{i}(\ell_{i})\ne0$ if $\e(i)=1$. Then $I^{*}_{\e}$ is  a $\prod_{i=1}^{m}\SL(N_{i})$-orbit with unipotent stabilizer, which corresponds to a $B_{I}$-orbit on $I^{*}\times O$ with unipotent stabilizer.
\end{exam}

\begin{exam}\label{ex:mot} In \cite{Y-motive}, we constructed a geometrically rigid \aud\ for $X=\PP^{1}_{k}$ and $S=\{0,1,\infty\}$ for groups $G$ of type $A_{1}, D_{2n}, E_{7}, E_{8}, G_{2}$ that only involve mutiplicative characters. One interesting observations is that the construction makes sense over any field $k$ with $\chk\ne2$.  In particular, in \cite{Y-motive} it was applied to $k=\QQ$ to obtain the first examples of motives with $E_{7}$ and $E_{8}$ as motivic Galois groups, and to solve the inverse Galois problem for $E_{8}(\FF_{\ell})$ for sufficiently large prime $\ell$.

Assume $G$ is simply-connected, and  the longest element $w_{0}$ in the Weyl group $W$ of $G$ acts by inversion on $T$.



Up to conjugacy, there is a unique parahoric subgroup $\bP\subset L_{0}G$ such that its reductive quotient $L_{\bP}$ is isomorphic to the fixed point subgroup $G^{\tau}$ of a Cartan involution corresponding to the split real form of $G$. For example, we can take $\bP$ to be the parahoric subgroup corresponding to the facet containing the element $\rho^{\vee}/2$ in the $T$-apartment of the building of $L_{0}G$ ($\rho^{\vee}$ is half the sum of positive coroots of $G$). 

The Dynkin diagram of the reductive quotient $L_{\bP}\cong G^{\tau}$ of $\bP$ is obtained by removing one or two nodes from the extended Dynkin diagram of $G$. We tabulate the type of $L_{\bP}$ and the nodes to be removed in each case.

\begin{tabular}{|c|c|c|}
\hline
$G$ & $L_{\bP}$ & nodes to be removed\\ \hline
$B_{2n}$ & $B_{n}\times D_{n}$ & the $(n+1)$-th counting from the short node \\ \hline
$B_{2n+1}$ & $B_{n}\times D_{n+1}$ & the $(n+1)$-th counting from the short node \\ \hline
$C_{n}$ & $A_{n-1}\times\Gm $ & the two ends \\ \hline
$D_{2n}$ & $D_{n}\times D_{n}$ & the middle node \\ \hline 
$E_{7}$ & $A_{7}$ & the end of the leg of length 1 \\ \hline
$E_{8}$ & $D_{8}$ & the end of the leg of length 2 \\ \hline
$F_{4}$ & $A_{1}\times C_{3}$ & second from the long node end\\ \hline
$G_{2}$ & $A_{1}\times A_{1}$ & middle node \\ \hline
\end{tabular}

{\bf Fact}:  If $G$ is not of type $C$, then $L^{\sc}_{\bP}\to L_{\bP}$ is a double cover. Even if $G$ is of type $C_{n}$, $L_{\bP}\cong\GL_{n}$ still admits a unique nontrivial double cover.

Therefore, in all cases, there is a canonical nontrivial double cover $v:\wt{L}_{\bP}\to L_{\bP}$. In particular,
\begin{equation}\label{double L}
(v_{!}\Qlbar)_{\textup{sgn}}\in \CS(L_{\bP})
\end{equation}
(here $\textup{sgn}$ denotes the nontrivial character of $\ker(v)(k)=\{\pm1\}$). 

Let $\bK_{0}=\bP\subset L_{0}G$ be a parahoric subgroup of the type defined above. Let $\cK_{0}$ be the pullback of the local system \eqref{double L}. Let $\bK_{\infty}=\bP_{\infty}\subset L_{\infty}G$ be the parahoric subgroup of the same type as $\bP$. Let $\bK_{1}=\bI_{1}\subset L_{1}G$ be an Iwahori subgroup. Let $\cK_{\infty}$ and $\cK_{1}$ be trivial. 

The central character compatible with $(K_{S}, \chi_{S})$ above is unique, and it exists if and only if $\chi|_{Z(k)}=1$, which can always be achieved by passing to a quadratic extension $k'/k$.

In \cite{Y-motive}, it is shown that if $G$ is either simply-laced or of type $G_{2}$ (again assume $w_{0}=-1$, so $G$ is of type $A_{1}, D_{2n}, E_{7}, E_{8}$ and $G_{2}$), then the geometric \aud\ $(\cK_{S},\cK_{S})$ defined above is weakly rigid. Indeed, there is a unique relevant $\ov k$-point in $\Bun_{G}(\bK_{S})$ which is the unique open point. 

The moduli stack $\Bun_{G}(\bP_{0},\bP_{\infty})$ contains an open substack isomorphic to $\pt/L_{\bP}$. The preimage of $\pt /L_{\bP}$ in $\Bun_{G}(\bK_{S})$ is isomorphic to $L_{\bP}\backslash G/B$. Since $L_{\bP}$ is a symmetric subgroup of $G$, it acts on the flag variety of $G$ with an open orbit $O\subset  G/B$. We thus get an open point
\begin{equation*}
j: L_{\bP}\backslash O\incl \Bun_{G}(\bK_{S}).
\end{equation*}
When $G$ is either simply-laced or of type $G_{2}$, this turns out to be the unique relevant point.

The stabilizer $A_{u}$ of $L_{\bP}$ on any point $u\in O$ is canonically isomorphic to $T[2]$. The category $D_{c}(\bK_{S},\cK_{S})$ is more interesting. Taking the preimage of $A_{u}$ in the double cover $\tilL_{\bP}$ we get a central extension
\begin{equation*}
1\to \mu_{2}\to \tilA_{u}\to T[2]\to1.
\end{equation*}
For simplicity let's assume $G=E_{8}$. Then $\tilA_{u}$ is a finite Heisenberg $2$-group with $\mu_{2}$ as the center. It has a unique irreducible $\Qlbar$-representation $V$ with nontrivial central character. This representation gives an irreducible local system on the preimage of the open point $L_{\bP}\backslash O$ in $\Bun_{G}(\bK^{+}_{S})$, whose extension to $\Bun_{G}(\bK^{+}_{S})$ by zero gives an object $\calF\in D_{c}(\kbar;\bK_{S},\calK_{S})$. It turns out $\cF$ is the unique simple perverse sheaf in $D_{c}(\kbar;\bK_{S},\calK_{S})$. 

When $G$ is not of type $E_{8}$, the center of $\tilA_{u}$ is larger. One can construct a simple  perverse sheaf for each central character of $\tilA_{u}$ whose restriction to $\mu_{2}$ is nontrivial. 
\end{exam}

\section{Computing eigen local systems}\label{s:Hk}

In this section we explain how to get $\dG$-local systems out of  rigid \auda.

\subsection{From eigenforms to local systems}\label{ss:global Langlands}
We give a very brief review of how Langlands correspondence works over function fields in the automorphic to Galois direction. 

\sss{Eigenvalues}
Given a Hecke eigenform $f$, then for $x\in |X|-S$ (where $S$ is a finite set of places), we have a character $\s_{x}: \cH_{G(\cO_{x})}\to \Qlbar$  of the spherical Hecke algebra $\cH_{G(\cO_{x})}$ such that 
\begin{equation*}
f\star h_{x}=\s_{x}(h_{x})f, \quad \forall h_{x}\in \cH_{G(\cO_{x})}.
\end{equation*}

We recall the Satake isomorphism:
\begin{equation*}
\cH_{G(\cO_{x})}\cong \Qlbar[\xcoch(T)]^{W}.
\end{equation*}
Each homomorphism $\s_{x}: \cH_{G(\cO_{x})}\to \Qlbar$ gives rise to a $W$-orbit on $\dT$, which is the same datum as a semisimple conjugacy class $[\phi_{x}]\in \dG$.  What property does the collection $\{\phi_{x}\}_{x\in |X|-S}$ satisfy?

In the simplest case $G=\PGL_{2}$ and $\dG=\SL_{2}$, $\s_{x}$ is determined by its value at the characteristic function of $G(\cO_{x})\mat{t_{x}}{0}{0}{1}G(\cO_{x})$; this value is equal to the number $a_{x}=\Tr(\phi_{x})$.  The $\{a_{x}\}$ is the function field analogues of Fourier coefficients $a_{p}$ for a Hecke eigen modular form for $\PGL_{2}$ over $\QQ$. Can we get arbitrary collection of traces $\{a_{x}\}_{x\in |X|-S}$ from eigenforms?

The Langlands correspondence predicts that the collection $\{\phi_{x}\}$ of conjugacy classes in $\dG$ must come from a single global object, namely a $\dG$-local system on $X-S$. 

\sss{Local systems} A rank $n$ $\Qlbar$-local system on $U$ for the \'etale topology is slightly technical to define: it arises from an inverse system of locally constant $\ZZ/\ell^{m}\ZZ$ sheaves on $U$ free of rank $n$.  Let $\Loc(U,\Qlbar)$ be the tensor category of $\Qlbar$-local systems on $U$  of varying rank. This is a tensor category linear over $\Qlbar$.

Let $u\in U$ be a geometric point $u\in U$. We have the \'etale fundamental group $\pi_{1}(U,u)$, which is a profinite group. A rank $n$ local system $E$ on $U$ with $\Qlbar$-coefficients gives rise to a continuous representation of $\pi_{1}(U,u)$ on a $n$-dimensional $\Qlbar$-vector space (with $\ell$-adic topology), called the monodromy representation of $E$. This representation is constructed by looking at the action of $\pi_{1}(U,u)$ on the stalk $E_{u}$.

The above discussion gives an equivalence of tensor categories
\begin{equation}\label{fiber u}
\om_{u}: \Loc(U,\Qlbar)\isom \Rep_{\cont}(\pi_{1}(U,u),\Qlbar).
\end{equation}

Note that $\pi_{1}(U,u)$ is a quotient of the absolute Galois group $\Gal(F^{\sep}/F)$ of the function field $F=k(X)$. Therefore, a rank $n$ local system on $U$ gives rise to an $n$-dimensional continuous representation of $\Gal(F^{\sep}/F)$.


\sss{$\dG$-local systems} With the choice of a geometric point $u\in U$, a $\dG$-local system on $U$ is a continuous homomorphism
\begin{equation}\label{rho to dG}
\rho:\pi_{1}(U,u)\to \dG(\Qlbar).
\end{equation}
Therefore, a $\GL_{n}$-local system is the same datum as a rank $n$ local system on $U$. 

A more canonical definition can be given using the Tannakian formalism.  Let $\Rep(\dG,\Qlbar)$ be the tensor category of algebraic representations of $\dG$ on finite-dimensional $\Qlbar$-vector spaces.  Then a $\dG$-local system on $U$ is the same datum as a tensor functor
\begin{equation}\label{ten loc}
E: \Rep(\dG,\Qlbar)\to \Loc(U,\Qlbar).
\end{equation}

The two notions of $H$-local systems are equivalent. Given a representation $\rho$ as in \eqref{rho to dG}, we define a tensor functor $E$ by assigning to $V\in\Rep(H,\Qlbar)$ the local system with monodromy representation
\begin{equation*}
\rho_{V}:\pi_{1}(U,u)\xrightarrow{\rho}\dG(\Qlbar)\to\GL(V).
\end{equation*}
Conversely, given a tensor functor $E$ as in \eqref{ten loc}, using the equivalence \eqref{fiber u}, $E$ can be viewed as a tensor functor $\Rep(\dG,\Qlbar)\to \Rep_{\cont}(\pi_{1}(U,u),\Qlbar)$. Tannakian formalism \cite{DM} then implies that such a tensor functor comes from a group homomorphism $\rho$ as in \eqref{rho to dG}, well-defined up to conjugacy. 

\begin{exam} For $\dG=\SO_{n}$, a $\dG$-local system on $U$ is the same thing as a rank $n$ local system $E$ on $U$ together with a self-adjoint isomorphism $b:E\isom E^{\vee}$.\end{exam}

Below we will get more technical because we will be using sheaf theory on algebraic stacks such as $\Bun_{G}$ and its Hecke correspondences to compute the $\dG$-local systems attached to an eigenform. The ideas come from the geometric Langlands program originated from the works of Drinfeld, Deligne, Laumon, etc. The main result is Theorem \ref{th:eigencat}, which roughly says that for rigid \aud, one can construct the corresponding local system explicitly.


\subsection{The Satake category}\label{Sat} 
The Satake category is an upgraded version of the spherical Hecke algebra $\cH_{G(\calO_{x})}$ under the sheaf-to-function correspondence. In this subsection we  $G$ is split reductive.

Let $LG$ and $L^{+}G$ be the group objects over $k$ whose $R$-points are $G(R\lr{t})$ and $G(R\tl{t})$ respectively.  The quotient $\Gr=LG/L^{+}G$ is  called the {\em affine Grassmannian} of $G$, and it is an infinite union of projective schemes.  Then $L^{+}G$ acts on $\Gr$ via left translation. The $L^{+}G$-orbits on $\Gr$ are indexed by  dominant coweights $\l\in\xcoch(T)^+$. The orbit containing the element $t^{\l}\in T(k\lr{t})$ is denoted by $\Gr_{\l}$ and its closure is denoted by $\Gr_{\leq\l}$. We have $\dim\Gr_{\l}=\j{2\rho,\l}$, where $2\rho$ is the sum of positive roots in $G$. Each $\Gr_{\leq\l}$ is a projective but usually singular variety over $k$. We denote the intersection complex of $\Gr_{\leq\l}$ by $\IC_{\l}$.


Let $\Sat_{\ov k}=\Perv_{(L^{+}G)_{\ov k}}(\Gr_{\ov k})$ be the category of $(L^{+}G)_{\ov k}$-equivariant perverse sheaves on $\Gr_{\ov k}=\Gr\ot_{k}\ov k$ (with $\Qlbar$-coefficients) supported on finitely many $(L^{+}G)_{\ov k}$ orbits. In \cite{Lu}, \cite{Ginz} and \cite{MV}, it was shown that when $k$ is algebraically closed, $\Sat_{\ov k}$ carries a natural tensor structure, such that the global cohomology functor $h=H^*(\Gr_{\ov k},-):\Sat_{\ov k}\to\Vect(\Qlbar)$ is a fiber functor. It is also shown that the Tannakian group of the tensor category $\Sat_{\ov k}$ is the Langlands dual group $\dG$. The Tannakian formalism gives the {\em geometric Satake equivalence} of tensor categories
\begin{equation*}
\Sat_{\ov k}\cong\Rep(\dG,\Qlbar).
\end{equation*}

Similarly we define $\Sat_{k}= \Perv_{L^{+}G}(\Gr)$; this is also a tensor category with fiber functor $H^*(\Gr_{\ov k},-)$, but its Tannakian group is larger than $\dG$ (one can show its Tannakian group is the algebraic envelope of $\dG\times\Gk$). Define $\Sat^{0}_{k}\subset\Sat_{k}$
to be the full subcategory consisting of direct sums of the {\em normalized} $\IC_{\l}$'s whose restriction to $\Gr_{\l}$ is the Tate twisted sheaf $\Qlbar[\j{2\rho,\l}](\j{\rho, \l})$ (it requires choosing a square root of $q$ in $\Qlbar^{\times}$ if $\rho\notin \xch(T)$).  It turns out that $\Sat^{0}_{k}$ is closed under the tensor structure, and we have an equivalence of tensor categories
\begin{equation*}
\Sat^{0}_{k}\cong\Rep(\dG,\Qlbar).
\end{equation*}
For $V\in\Rep(\dG,\Qlbar)$, we denote the corresponding object in $\Sat_{k}^{0}$ by $\IC_{V}$.

\begin{exam} For $G=\GL_{n}$, $\Gr(k)$ parametrizes $\cO=k\tl{t}$-lattices in $k\lr{t}^{n}$. For  $\l=(1,\cdots, 1,0,\cdots, 0)$ with $i$ $1$'s and $n-i$ 0's, $\Gr_{\l}=\Gr_{\le \l}\cong \Gr(n,i)$. In fact, $\Gr_{\l}(k)$ consists of lattices $\L$ such that $t^{-1}\cO^{n}\supset \L\supset \cO^{n}$ and $\dim_{k}(\L/\cO^{n})=i$.
\end{exam}

\begin{exam} For $G=\Sp(V)$, $V$ a symplectic space of dimension $2n$, $\Gr(k)$ parametrizes $\cO=k\tl{t}$-lattices $\L$ in $V\ot k\lr{t}$ such that the symplectic form restricts to a perfect pairing on $\L$ (these are called self-dual lattices).  Let $\L_{0}=V\ot\cO$ be the standard self-dual lattice. For $\l=(1,0,\cdots,0)\in\ZZ^{n}$, $\Gr_{\l}(k)$ consists of self-dual lattices $\L$ such that $\dim_{k}(\L/\L\cap \L_{0})= 1$. The orbit closure $\Gr_{\le\l}=\Gr_{\l}\cup \{\L_{0}\}$. There is a map $\Gr_{\l}\to \PP(V)$ sending $\L$ to the line in $V$ that is the image of $\L\mod t^{-1}\subset t^{-1}V$. The fiber over a point $\ell\in \PP(V)$ is the line $\Hom(\ell, V/\ell^{\bot})\cong \ell^{\ot -2}$ (where $\ell^{\bot}$ is defined using the symplectic form on $V$). Therefore $\Gr_{\l}$ is isomorphic to the total space of $\cO(2)$ over $\PP^{2n-1}$.
\end{exam}

\begin{ex} For $G=\SO(V)$ with $\dim V=2n$ or $2n+1$, describe $\Gr_{\l}$ for $\l=(1,0,\cdots, 0)\in \ZZ^{n}$.
\end{ex}

\subsection{Geometric Hecke operators} We consider the situation of \S\ref{sss:auto sheaves}. In particular, we have a geometric automorphic datum $(\bK_{S},\calK_{S})$, and moduli stacks $\Bun_{G}(\bK_{S})$ and $\Bun_{G}(\bK^{+}_{S})$.

\subsubsection{Hecke correspondence} Consider the following diagram
\begin{equation}\label{diag Hk}
\xymatrix{& \Hk_{G}(\bK^{+}_{S})\ar[dl]_{\oll{h}}\ar[dr]^{\orr{h}}\ar[rr]^{\pi} & & U:=X-S \\
\Bun_{G}(\bK^{+}_{S}) & & \Bun_{G}(\bK^{+}_{S})}
\end{equation}
Here, the stack $\Hk_{G}(\bK^{+}_{S})$ classifies the data $(x, \calE,\calE', \tau)$ where $x\in U:=X-S$, $\calE,\calE'\in\Bun_{G}(\bK^{+}_{S})$ and $\tau:\calE|_{X-\{x\}}\isom\calE'|_{X-\{x\}}$ is an isomorphism of $G$-torsors over $X-\{x\}$ preserving the $\bK^{+}_{x}$-level structures at each $x\in S$. The morphisms $\oll{h}$, $\orr{h}$ and $\pi$ send $(x,\calE,\calE',\tau)$ to $\calE,\calE'$ and $x$ respectively.

For $x\in U$, we denote its preimage under $\pi$ by $\Hk_{G,x}(\bK^{+}_{S})$. We have an evaluation morphism
\begin{equation}\label{evx}
\ev_{x}: \Hk_{G,x}(\bK^{+}_{S})\to L^{+}_{x}G\backslash L_{x}G/L^{+}_{x}G.
\end{equation}
In fact, for a point $(x,\calE,\calE',\tau)\in\Hk_{G,x}(\bK^{+}_{S})$, if we fix trivializations of $\calE$ and $\calE'$ over $\Spec\calO_{x}$, the isomorphism $\tau$ restricted to $\Spec F_{x}$ is an isomorphism between the trivial $G$-torsors over $\Spec F_{x}$, hence given by a point $g_{\tau}\in L_{x}G$. Changing the trivializations of $\calE|_{\Spec\calO_{x}}$ and $\calE'|_{\Spec\calO_{x}}$ will result in left and right multiplication of $g_{\tau}$ by elements in $L_{x}^{+}G$. Therefore we have a well-defined morphism $\ev_{x}$ as above between stacks.

As $x$ moves along $U$, we may identify the target of \eqref{evx} as $L^{+}G\backslash LG/L^{+}G$ by choosing a local coordinate $t$ at $x$. Modulo the ambiguity caused by the choice of the local coordinates, we obtain a well defined morphism 
\begin{equation}\label{ev Hk}
\ev: \Hk_{G}(\bK^{+}_{S})\to \left[\frac{L^{+}G\backslash LG/L^{+}G}{\Aut^{+}}\right],
\end{equation}
where $\Aut^{+}$ is the group scheme over $k$ of continuous ring automorphisms of $k\tl{t}$, and it acts on $LG$ and $L^{+}G$ via its action on $k\tl{t}$.

\subsubsection{Geometric Hecke operators}
For each object $V\in\Rep(\dG,\Qlbar)$, the corresponding object $\IC_{V}\in\Sat$ under the geometric Satake equivalence defines a complex on the quotient stack $\left[\frac{L^{+}G\backslash LG/L^{+}G}{\Aut^{+}}\right]$.  The geometric Hecke operator associated with $V$ is the functor
\begin{eqnarray*}
\TT_{V}: D_{(L_{S},\calK_{S})}(\Bun_{G}(\bK^{+}_{S})\times U)&\to& D_{(L_{S},\calK_{S})}(\Bun_{G}(\bK^{+}_{S})\times U)\\
\calF &\mapsto& (\orr{h}\times\pi)_{!}\left((\oll{h}\times\pi)^{*}\calF\otimes\ev^{*}\IC_{V})\right).
\end{eqnarray*}

The composition of these functors are compatible with the tensor structure of $\Sat$: there is a natural isomorphism of functors
\begin{equation*}
\TT_{V}\circ\TT_{W}\cong \TT_{V\otimes W}, \quad \forall V,W\in\Rep(\dG,\Qlbar)
\end{equation*}
which is compatible with the associativity constraint of the tensor product in $\Rep(\dG,\Qlbar)$ and the associativity of composition of functors $\TT_{V_{1}}\circ\TT_{V_{2}}\circ\TT_{V_{3}}$ in the obvious sense.

\begin{theorem}[{\cite[Corollary A.4.2]{JY}}, imprecise statement]\label{th:eigencat} Let $(\bK_{S}, \cK_{S})$ be a geometrically rigid \gaud.  Then one can decompose $\Perv_{c}(\ov k;\bK_{S}, \cK_{S})$ (perverse sheaves in $D_{c}(\ov k;\bK_{S}, \cK_{S})$) into a finite direct sum of subcategories stable under the geometric Hecke operators, such that each direct summand gives rise to a $\dG$-local system on $U$.
\end{theorem}

The proof uses the structure of an $E_{2}$-action of $\Rep(\dG)$ on $\Perv_{c}(\ov k;\bK_{S}, \cK_{S})$ that ``spreads over'' the curve $U$. Such a structure comes from versions of geometric Hecke operators modifying at two or more points on $U$.

A special case is when $\Perv_{c}(\ov k;\bK_{S}, \cK_{S})$ has a single simple object $\cF$ (which we may assume is defined over $k$). In this case, Theorem \ref{th:eigencat} says the following: $\cF$  is a {\em Hecke eigensheaf} in the following sense. For every $V\in \Rep(\dG,\Qlbar)$, there is an isomorphism
\begin{equation*}
\varphi_{V}: \TT_{V}(\calF\boxtimes {\Qlbar})\cong\calF\boxtimes E_{V}
\end{equation*}
for some local system $E_{V}\in \Loc(U,\Qlbar)$. Moreover the assignment $V\mapsto E_{V}$ gives a tensor functor $E:\Rep(\dG,\Qlbar)\to \Loc(U)$, hence a $\dG$-local system $E$.

\subsection{Computing local systems --  a simple case}\label{ss:comp loc simple} 
Fix a geometric automorphic datum $(\bK_{S}, \calK_{S})$.
\sss{Assumptions} We assume $G$ is simply-connected and $(\bK_{S},\cK_{S})$ is geometrically rigid. Since $\Bun_{G}(\bK_{S})$ is connected, it has a unique relevant point $\cE$. We further assume that $\Aut(\cE)$ is trivial. This implies that $\cE$ is an open point in $\Bun_{G}(\bK_{S})$.

\sss{} Let $\frG$ be the group ind-scheme over $U$ whose fiber over $x\in U$ is the automorphism group of $\cE|_{X-\{x\}}$ as a $G$-torsor with $\bK_{S}$-level structures. The evaluation map along $S$ (well-defined up to conjugacy) can be extended to $\frG$:
\begin{equation*}
\ev_{S,\cE}: \frG\to \prod_{x\in S}\bK_{x}\ot \ov k\to L_{S}\ot \ov k.
\end{equation*}

For any dominant coweight $\l\in\xcoch(T)^{+}$, let $\frG^{\le \l}\subset \frG$ be the closed subscheme whose fiber over $x\in U$ are those automorphisms of $\cE|_{X-\{x\}}$ with modification type $\le \l$ at $x$. Thus $\frG^{\le \l}_{x}$ can be identified with an open subscheme of $\Gr_{\le \l}$.  

\begin{prop}\label{p:loc simple} Under the assumption that $G$ is simply-connected and $\Aut(\cE)$ is trivial, $\Perv_{c}(\ov k; \bK_{S}, \cK_{S})$ contains a unique simple object $\cF$, and it is a Hecke eigensheaf. The corresponding $\dG$-local system $E$ is described as follows.

For a  dominant coweight $\l\in\xcoch(T)^{+}$ and the corresponding irreducible representation $V_{\l}$ of $\dG$, we have
\begin{equation*}
E_{V_{\l}}\cong \pi^{\l}_{!}(\ev_{S,\cE}^{*}\cK_{S}\ot \IC_{\frG^{\le \l}}[-1])
\end{equation*}
where $\pi^{\l}: \frG^{\le \l}\to U$ is the projection (it contains the implicit statement that the right side is concentrated in degree zero).
\end{prop}

In particular, if $\l$ is minuscule so that $\frG^{\le\l}=\frG^{\l}$ is smooth of relative dimension $\j{\l,2\r}$ over $U$, then
\begin{equation*}
E_{V_{\l}}\cong \bR^{\j{\l,2\r}}\pi^{\l}_{!}\ev_{S,\cE}^{*}\cK_{S}.
\end{equation*}

\subsection{Computing local systems -- general case}
Now we drop the condition that $G$ be simply-connected (still assumed to be split semisimple) and the automorphism groups of the relevant points be trivial. The discussion in the general case is a bit technical.

\subsubsection{Twisted representations}\label{ss:tw rep} Let $\Gamma$ be a group and $\xi\in Z^{2}(\Gamma, \Qlbar^{\times})$ a cocycle such that $\xi_{1,\gamma}=\xi_{\gamma,1}=1$ for all $\gamma\in\Gamma$. A $\xi$-twisted representation of $\Gamma$ is a finite-dimensional $\Qlbar$-vector space $V$ with automorphisms $T_{\gamma}:V\to V$, one for each $\gamma\in\Gamma$, such that $T_{1}=\id_{V}$ and
\begin{equation*}
T_{\gamma\delta}=\xi_{\gamma,\delta}T_{\gamma}T_{\delta}, \forall\gamma,\delta\in\Gamma.
\end{equation*}
Let $\Rep_{\xi}(\Gamma,\Qlbar)$ be the category of $\xi$-twisted representations of $\Gamma$. This is a $\Qlbar$-linear abelian category which, up to equivalence, only depends on the cohomology class $[\xi]\in \cohog{2}{\Gamma, \Qlbar^{\times}}$.

A natural source of $2$-cocycles on $\G$ come from rank one character sheaves. We did not define rank one character sheaves for disconnected groups but the same definition works, except now they may have nontrivial automorphisms. Isomorphism classes in $\CS(\G)$ are in bijection with $\cohog{2}{\Gamma, \Qlbar^{\times}}$.

Suppose  $\calE\in\Bun_{G}(\bK_{S})(\kbar)$ be a relevant point for $(\bK_{S},\cK_{S})$. Let $A_{\cE}=\Aut(\cE)$ and $\G_{\cE}=\pi_{0}(\Aut(\calE))$. Since $\ev_{S,\ \calE}^{*}\calK_{S}$ is trivial on $\Aut^{\c}(\calE)$,  it descends to a rank one character sheaf on $\G_{\cE}$ and gives a cocycle $\xi\in Z^{2}(\G_{\cE}, \Qlbar^{\times})$ satisfying $\xi_{1,\g}=\xi_{\g,1}=1$ for all $\g\in \G_{\cE}$, whose cohomology class is well-defined.  

\begin{lemma}\label{l:relpt twrep} Let $\Perv_{c}(\ov k; \bK_{S},\cK_{S})_{\cE}$ be the category of perverse sheaves in $\Perv_{c}(\ov k; \bK_{S},\cK_{S})$ that have vanishing stalks outside the preimage of the relevant point $\cE$. Then  $\Perv_{c}(\ov k; \bK_{S},\cK_{S})_{\cE}\cong \Rep_{\xi}(\G_{\cE},\Qlbar)$.
\end{lemma}
\begin{proof}[Sketch of proof] We base change the spaces to $\ov k$ without changing notation. We can find a finite isogeny $\nu: \wt L_{S}\to L_{S}$ with discrete kernel $C$, and a character $\chi_{C}: C\to \Qlbar^{\times}$  such that the local system  $\cK_{S}$ on $L_{S}$ is of the form
\begin{equation*}
\cK_{S}\cong(\nu_{*}\Qlbar)_{\chi_{C}}\in \CS(L_{S}).
\end{equation*}
Let $\wt A_{\cE}$ be the pullback of the cover $\nu$ along $\ev_{S,\cE}: A_{\cE}\to L_{S}$. Let $\wt\G_{\cE}=\pi_{0}(\wt A_{\cE})$. Then $\G_{\cE}$ fits into an exact sequence
\begin{equation*}
1\to C\to \wt\G_{\cE}\to \G_{\cE}\to 1.
\end{equation*}
The pushout of the sequence along $\chi_{C}: C\to\Qlbar^{\times}$ gives a $2$-cocycle $\xi\in Z^{2}(\G_{\cE},\Qlbar^{\times})$ (upon choosing a set-theoretic splitting $s: \G_{\cE}\to \wt\G_{\cE}$, so well-defined up to coboundaries).  Let $\Rep^{\chi_{C}}(\wt\G_{\cE}, \Qlbar)$ be the category of finite-dimensional $\Qlbar$-representations of $\wt\G_{\cE}$ whose restriction to $C$ is $\chi_{C}$.. We have an equivalence
\begin{equation*}
\Rep^{\chi_{C}}(\wt\G_{\cE}, \Qlbar)\isom \Rep_{\xi}(\G_{\cE},\Qlbar)
\end{equation*}
by restriction along $s$.

The preimage of $\cE$ in $\Bun_{G}(\bK_{S}^{+})$ is isomorphic to $L_{S}/A_{\cE}\cong \wt L_{S}/\wt A_{\cE}$ where $A_{\cE}$ acts on $L_{S}$ via $\ev_{S,\cE}$ and right translation. For $V\in \Rep_{\xi}(\G_{\cE},\Qlbar)$ viewed as a representation of $\wt \G_{\cE}$ with restriction $\chi_{C}$ to $C$,  we get a $\wt L_{S}$-equivariant local system $\cF_{V}$ on $\wt L_{S}/\wt A_{\cE}\cong L_{S}/A_{\cE}$ (since $\wt\G_{\cE}=\pi_{0}(\wt A_{\cE})$). Let $i: L_{S}/A_{\cE}\incl \Bun_{G}(\bK_{S},\cK_{S})$ be the inclusion, then up to a shift $i_{!}\cF_{V}$ is the object in $\Perv_{c}(\ov k; \bK_{S},\cK_{S})_{\cE}$ corresponding to $V$.
\end{proof}

\sss{General case} Let 
$$\Bun_{G}(\bK_{S})=\coprod_{\a\in \Om}\Bun_{G}(\bK_{S})_{\a}$$ be the decomposition into connected components. Here $\Om=\xcoch(T)/\ZZ\Phi^{\vee}$ is the algebraic $\pi_{1}$ of $G$. Let $D_{c}(\ov k; \bK_{S}, \cK_{S})_{\a}\subset D_{c}(\ov k;\bK_{S}, \cK_{S})$ be the full subcategory of sheaves supported on $\Bun_{G}(\bK_{S})_{\a}$. Similarly define the abelian category of perverse sheaves $\Perv_{c}(\ov k; \bK_{S}, \cK_{S})_{\a}$.

Let $\cE_{\a}$ be the unique relevant $\kbar$-point in $\Bun_{G}(\bK_{S})_{\a}$. Let $A_{\a}=\Aut(\cE_{\a})$ (an algebraic  group over $\ov k$), and $
\G_{\a}=\pi_{0}(A_{\a})$ as a discrete group. 

Define ${}_{\b}\frG_{\a}\to U$ whose fiber over $x\in U$ parametrize isomorphisms of $G$-torsors $\calE_{\a}|_{X-\{x\}}\isom \cE_{\b}|_{X-\{x\}}$ preserving the $\bK_{S}$-level structures. This is a right torsor under the group ind-scheme ${}_{\a}\frG_{\a}$ and a left torsor under ${}_{\b}\frG_{\b}$.  In particular, $A_{\a}$ acts on ${}_{\b}\frG_{\a}$ from the right and $A_{\b}$ acts from the left. Upon choosing trivializations of $\cE_{\a}$ and $\cE_{\b}$ near $S$, we have the evaluation map
\begin{equation*}
{}_{\a}\ev_{S, \b}: {}_{\a}\frG_{\b}\to \prod_{x\in S}\bK_{x}\ot \ov k\to L_{S}\ot\ov k.
\end{equation*}


Choose $(\wt L_{S}, \chi_{C})$ as in the proof of Lemma \ref{l:relpt twrep}. Pulling back the covering $\nu: \wt L_{S}\to L_{S}$ along ${}_{\a}\ev_{S,\b}$ we get $C$-torsors ${}_{\a}\tfrG_{\b}\to {}_{\a}\frG_{\b}$ and central extensions $1\to C\to \wt A_{\a}\to A_{\a}\to 1$ and $1\to C\to  \wt \G_{\a}\to \G_{\a}\to 1$.

Then by Lemma \ref{l:relpt twrep} we have an equivalence of categories
\begin{equation}\label{perv a}
\Perv_{c}(\ov k; \bK_{S}, \cK_{S})_{\a}\cong \Rep_{\xi_{\a}}(\G_{\a},\Qlbar)\cong \Rep^{\chi_{C}}(\wt\G_{\a},\Qlbar).
\end{equation}
where the last category consists of finite-dimensional $\Qlbar$-representations of $\wt\G_{\a}$ whose restriction to $C$ is $\chi_{C}$.

We have an evaluation map
\begin{equation*}
{}_{\a}\ev_{U,\b}: {}_{\a}\frG_{\b}\to \left[\dfrac{L^{+}G\bs LG/L^{+}G}{\Aut^{+}}\right]
\end{equation*}
recording the modification at  $x\in U$. Let ${}_{\a}\frG^{\le \l}_{\b}$ be the preimage of $L^{+}G\bs \Gr_{\le \l}$ under the above map, then ${}_{\a}\frG^{\le \l}_{\b}\ne\vn$ only when $\l$ has image $\b-\a\in \Om$. The pullback ${}_{\a}\ev_{U,\b}^{*}\IC_{\l}$ is supported on ${}_{\a}\frG^{\le \l}_{\b}$.  Let ${}_{\a}\wt\ev_{U,\b}$ be the precomposition of ${}_{\a}\ev_{U,\b}$ with ${}_{\a}\tfrG_{\b}\to {}_{\a}\frG_{\b}$. We also have ${}_{\a}\wt\ev_{U,\b}^{*}\IC_{\l}$ supported on ${}_{\a}\tfrG^{\le \l}_{\b}$.

Let ${}_{\a}\wt\pi_{\b}^{\l}: {}_{\a}\tfrG_{\b}\to U$ be the projection. Let $\Sig_{\a}$ be the set of irreducible object in $\Rep^{\chi_{C}}(\wt\G_{\a}, \Qlbar)$. For $\a\in\Om$, let $\calF_{\eta}\in \Perv_{c}(\ov k; \bK_{S}, \cK_{S})_{\a}$ corresponding to an irreducible object $\y\in \Sig_{\a}$ under \eqref{perv a}. 

\begin{prop}\label{p:loc} Let $(\bK_{S},\cK_{S})$ be geometrically rigid.  Then for a dominant coweight $\l$ with image $\b-\a$ in $\Om$, we have
\begin{equation*}
\TT_{V_{\l}}(\cF_{\y}\bt\Qlbar)\cong \op_{\z\in \Sig_{\b}}\cF_{\z}\bt {}_{\z}E^{\l}_{\y}
\end{equation*}
where ${}_{\z}E_{\y}^{\l}$ is the local system on $U$ given by
\begin{equation}\label{compute rho}
{}_{\z}E^{\l}_{\y}\cong \Hom_{\wt\G_{\a}\times\wt\G_{\b}}(\y\boxtimes \z^{\vee}, {}_{\a}\wt{\pi}^{\l}_{\b!}{}_{\a}\wt{\ev}_{\b}^{*}\IC_{\l})[-1].
\end{equation}
\end{prop}

If $\Aut(\cE_{\a})$ is trivial for all $\a$, there is a unique simple perverse sheaf $\cF_{\a}\in D_{c}(\ov k;\bK_{S},\cK_{S})_{\a}$. The sum of $\cF_{\a}$ is a Hecke eigensheaf whose eigen $\dG$-local system $E$ is described as follows.  For a dominant coweight $\l$ with image $\b-\a$ in $\Om$, we have
\begin{equation*}
E_{V_{\l}}\cong {}_{\a}\pi^{\l}_{\b!}({}_{\a}\ev_{S,\b}^{*}\cK_{S}\ot {}_{\a}\ev^{*}_{\b}\IC_{\l})[-1].
\end{equation*}
When $\l$ is minuscule we have
\begin{equation}\label{Emin ab}
E_{V_{\l}}\cong \bR^{\j{\l,2\r}}{}_{\a}\pi^{\l}_{\b!}{\a}\ev_{S,\b}^{*}\cK_{S}.
\end{equation}


%
%

\subsection{Kloosterman local systems}\label{ss:Kl}
We apply the method above to the Kloosterman \aud\ in Example \ref{ex:Kl} in the case $G=\PGL_{n}$. We consider only the case $\l=(1,0,\cdots, 0)$, so $V_{\l}$ is the standard representation of $\dG=\SL_{n}$. So we are looking for an $\dG=\SL_{n}$-local system on $U=\PP^{1}-\{0,\infty\}\cong \Gm$. The answer turns out to be a classical object.

\sss{Calculation of $\frG^{\le \l}$}
The moduli stack $\Bun_{\GL_{n}}(\bK_{S})$ classifies equivalence classes of $(\cE, F^{*}, F_{*}, )$ where 
\begin{enumerate}
\item $\cE$ is a rank $n$ vector bundle over $\PP^{1}$;
\item $F^{*}: (0=F^{n}\subset F^{n-1}\subset \cdots\subset F^{0}=\cE|_{0})$  is a  complete flag of the fiber of $\cE$ at $0$;
\item $F_{*}: (0=F_{0}\subset F_{1}\subset \cdots\subset F_{n}=\cE|_{\infty})$  is a  complete flag of the fiber of $\cE$ at $\infty$;
\item For $1\le i\le n$, $v_{i}\in F_{i}/F_{i-1}$ is a nonzero vector.
\end{enumerate}
Let $\Pic_{\infty}$ be the moduli space of line bundles on $\PP^{1}$ together with a trivialization at $\infty$. We have a canonical isomorphism $\Pic_{\infty}\cong \ZZ$ by taking degrees. Then $\Pic_{\infty}$ acts on $\Bun_{\GL_{n}}(\bK_{S})$ by tensoring, and $\Bun_{G}(\bK_{S})=\Bun_{\GL_{n}}(\bK_{S})/\Pic_{\infty}$. 

The components of $\Bun_{G}(\bK_{S})$ are parametrized by $\ZZ/n\ZZ$ by taking $\deg\cE$ mod $n$. Let $\cE_{i}$ be the unique relevant point of degree $i\in \ZZ_{n\ZZ}$. Then we have the following description:
\begin{enumerate}
\item $\cE_{0}=\cO e_{1}\op \cO e_{2}\op\cdots\op \cO e_{n}$,   $F^{i}=\Span\{e_{i+1}, \cdots, e_{n}\}$,  $F_{i}=\Span\{e_{1},\cdots, e_{i}\}$, $v_{i}=e_{i}\mod F_{i-1}$. 
\item $\cE_{1}=\cO e_{1}\op \cO e_{2}\op\cdots\op \cO(1) e_{n}$. Here $\cO(1)$ is equipped with a trivialization at $\infty$, so we identify it with $\cO(\{0\})$. Take $F^{i}=\Span\{e_{i}, \cdots, e_{n-1}\}$ if $i\ge 1$, $F_{i}=\Span\{e_{1},\cdots, e_{i}\}$,  and $v_{i}=e_{i}\mod F_{i-1}$. 
\end{enumerate}
A point in $\frG^{\le \l}$ is an injective map of coherent sheaves $f: \cE_{0}\to \cE_{1}$ preserving $F^{*}, F_{*}$ and $\{v_{i}\}$. We may write $f$ as a matrix $(a_{ij})$ where $a_{ij}\in k$ if $i<n$, and $a_{nj}\in \G(\PP^{1},\cO(1))=k\op k\t$. The fact that $f$ preserve $F^{*}, F_{*}$ and $\{v_{i}\}$ imply
\begin{enumerate}
\item $(a_{ij})|_{\t=0}$ is upper triangular with $1$ on the diagonal. This implies $(a_{ij})$ has the form
\begin{equation*}
\left(\begin{array}{cccc} 1 & * & * & * \\  0 &  \ddots & * & *  \\ 0 & 0 & 1 & * \\  a'_{n1}\t & \cdots & a'_{n,n-1}\t & 1+a'_{nn}\t \end{array}\right)
\end{equation*}

\item Let $a'_{ij}=a_{ij}$ if $i<n$ and let $a'_{nj}$ be as above. Then $(a'_{ij})$ sends the flag $e_{i}$ to $\Span\{e_{i-1},\cdots, e_{n-1}\}$ for $i=2,\cdots, n$.  This implies $(a_{ij})$ has the form 
\begin{equation*}
\left(\begin{array}{cccc} 1 & a_{1} & 0 & 0 \\  0 &  \ddots & * & 0  \\ 0 & 0 & 1 & a_{n-1} \\  a_{0}\t & \cdots & 0 & 1 \end{array}\right)
\end{equation*}
Moreover, $a_{0},\cdots, a_{n-1}$ are nonzero since they are the action of $f$ on the associated graded of the filtrations $F^{*}$ at $0$.
\end{enumerate}
We get an isomorphism $\frG^{\le \l}\cong \Gm^{n}$ with coordinates $(a_{0},\cdots,a _{n-1})$. The map $\pi: \frG^{\le \l}\to U$ sends $f$ to the support of the cokernel of $f$. Since $\det(f)=1+(-1)^{n-1}a_{0}\cdots a_{n-1}t^{-1}\in \G(\PP^{1}, \cO(1))$, we see that $\pi(f)=(-1)^{n}a_{0}a_{1}\cdots a_{n-1}$. 

For simplicity assume $\cK_{0}$ is trivial. By \eqref{Emin ab}, we get that ($\St=V_{\l}$ is the standard representation of $\dG=\SL_{n}$ )
\begin{equation}\label{KlSt}
E_{\St}\cong \bR^{n-1}\pi_{!}\phi^{*}\AS_{\psi}
\end{equation}
where we use $\phi$ to denote the composition $\frG^{\le \l}\cong \Gm^{n}\xr{\ev_{S}} \prod_{i=0}^{n-1} U_{\a_{i}}\xr{\ph}\Ga$; it is  linear in each coordinate $a_{i}$ with nonzero coefficient.

\subsubsection{The classical Kloosterman sheaf}
We first recall the definition of Kloosterman sums. Let $p$ be a prime number. Fix a nontrivial additive character $\psi:\FF_p\to\Qlbar^{\times}$. Let $n\geq2$ be an integer. Then the $n$-variable Kloosterman sum over $\FF_p$ is a function on $\FF^{\times}_{p}$ whose value at $a\in\FF_{p}^{\times}$ is
\begin{equation*}
\Kl_n(p;a)=\sum_{x_1,\cdots,x_n\in\FF^\times_p;x_1x_{2}\cdots x_{n}=a}\psi(x_1+\cdots+x_n).
\end{equation*}
These exponential sums arise naturally in the study of automorphic forms for $\GL_{n}$.

Deligne \cite{Deligne-ExpSum} gave a geometric interpretation of the Kloosterman sum. He considered the following diagram of schemes over $\FF_{p}$
\begin{equation*}
\xymatrix{& \Gm^n\ar[dl]_{\pi}\ar[dr]^{\ph}\\ \Gm & & \AA^1}
\end{equation*}
Here $\pi$ is the morphism of taking the product and $\ph$ is the morphism of taking the sum. He defines the {\em Kloosterman sheaf} to be 
\begin{equation*}
\Kl_n:=\bR^{n-1}\pi_!\ph^*\AS_\psi,
\end{equation*}
over $\Gm=\PP^{1}_{\FF_{p}}-\{0,\infty\}$. Up to a change of coordinates, this is essentially \eqref{KlSt}. The relationship between the local system $\Kl_{n}$ and the Kloosterman sum $\Kl_{n}(p;a)$ is explained by the following identity
\begin{equation*}
\Kl_n(p;a)=(-1)^{n-1}\Tr(\Frob_a,(\Kl_{n})_a).
\end{equation*}
Here $\Frob_{a}$ is the geometric Frobenius operator acting on the geometric stalk $(\Kl_{n})_{a}$ of $\Kl_{n}$ at $a\in\Gm(\FF_{p})=\FF_{p}^{\times}$.

\sss{Properties of Kloosterman local systems} The following properties of $\Kl_{n}$ were proved by Deligne. 
\begin{enumerate}
\item[(0)] $\Kl_n$ is a local system of rank $n$.
\item[(1)] $\Kl_n$ is tamely ramified at $0$, and the monodromy is unipotent with a single Jordan block.
\item[(2)] $\Kl_n$ is totally wild at $\infty$ (i.e., the wild inertia at $\infty$ has no nonzero fixed vector on the stalk of $\Kl_n$), and the Swan conductor $\Sw_\infty(\Kl_n)=1$. 
\end{enumerate}


Applying a special case of Theorem \ref{th:eigencat} to the Kloosterman \auda, we  get $\dG$-local systems $\Kl_{\dG}(\chi,\ph)$.  In \cite{HNY}, we show that $\Kl_{\dG}(\chi,\ph)$ enjoy analogous  properties as $\Kl_{n}$. These properties were predicted by Gross \cite{Gross-Prescribe}, Frenkel--Gross \cite{FG}. For example, when $\chi=1$ we prove:
\begin{enumerate}
\item $\Kl_{\dG}(1,\ph)$ is tame at $0$, and a generator of the tame inertia maps to a regular unipotent element in $\dG$.
\item The local monodromy of $\Kl_{\dG}(1,\ph)$ at $\infty$ is a {\em simple wild parameter} in the sense of Gross and Reeder \cite[\S5]{GR}.
\end{enumerate}

\section{Rigidity for local systems; Applications}

This section is likely not to be covered in the lectures. We compare the notion of rigidity for automorphic data and for local systems. We also mention some applications of the rigid \auda\ and some open problems.

%

\subsection{Rigid $\dG$-local systems}

We shall review the notion of rigidity for $\dG$-local systems, introduced by Katz \cite{Katz} for $\dG=\GL_{n}$. We assume the base field $k$ to be {\bf algebraically closed} in this subsection. Let $X$ be a complete smooth connected algebraic curve over $k$. Fix an open subset $U\subset X$ with finite complement $S$. 

\sss{Physical rigidity}
\begin{defn}[extending Katz {\cite[\S1.0.3]{Katz}}]\label{phy rig} Let $E$ be an $\dG$-local system on $U$. Then $E$ is called {\em physically rigid} if,  for any other $\dG$-local system $E'$,  $E'|_{\Spec F_{x}}\cong E|_{\Spec F_{x}}$ for all  $x\in S$ implies $E'\cong E$.
\end{defn}
Although the definition uses $U$ as an input, the notion of physical rigidity is in fact independent of the open subset $U$: for any nonempty open subset $V\subset U$, $\rho$ is rigid over $U$ if and only if $E|_{V}$ is rigid over $V$. Therefore, physical rigidity is a property of the Galois representation $\rho_{E}:\Gal(F^{\sep}/F)\to \dG(\Qlbar)$ obtained by restricting $E$ to a geometric generic point $\eta$ of the $X$.

\sss{Cohomological rigidity}\label{sss:coh rig}
Next we introduce cohomological rigidity. Let $\wh\frg$ be the Lie algebra of $\dG$, and let $\Ad(E)$ be local system $E_{\wh\frg}$, viewing $\wh\frg$ as the adjoint representation of $\dG$. Let $j:U\incl X$ be the open embedding and $j_{!*}\Ad(E)$ be the non-derived direct image of $\Ad(E)$ along $j$. Concretely, the stalk of $j_{!*}\Ad(E)$ at $x\in S$ is the $\calI_{x}$-invariants on $\wh\frg$. 

\begin{defn}[extending Katz {\cite[\S5.0.1]{Katz}}]\label{coho rig} A $\dG$-local system $E$ on $U$ is  called {\em cohomogically rigid}, if
\begin{equation*}
\t(E):=\cohog{1}{X,j_{!*}\Ad(E)}=0.
\end{equation*}
\end{defn}

The vector space $\t(E)$ does not change if we shrink $U$ to a smaller open subset. Therefore cohomological rigidity is also a property of the Galois representation $\rho_{E}:\Gal(F^{\sep}/F)\to \dG(\Qlbar)$.

\begin{remark}When we work over the base field $\CC$ and view $U$ as a topological surface, one can define a moduli stack $\calM$ of $\dG$-local systems over $U$ with prescribed local monodromy around the punctures $S$. Then $\t(E)$ is the Zariski tangent space of $\calM$ at $E$. The condition $\t(E)=0$ in this topological setting says that $E$ does not admit infinitesimal deformations with prescribed local monodromy around  $S$. This interpretation is the motivation for Definition \ref{coho rig}. 
\end{remark}

\begin{remark} An alternative approach to define the notion of rigidity for a local system $E$ over $U$ over a finite field $k$ is by requiring that the adjoint $L$-function of $E$ to be trivial (constant function $1$). This is the approach taken by Gross in \cite{Gross-Trivial L}. When $\cohog{0}{U_{\kbar},\Ad(E)}=0$, triviality of the adjoint $L$-function of $E$ is equivalent to cohomological rigidity of $E$.  
\end{remark}

Using the Grothendieck-Ogg-Shafarevich formula,  it is easy to give the following numerical criterion for cohomological rigidity:  $E$ is cohomologically rigid if and only if
\begin{equation}\label{num rigid}
\frac{1}{2}\sum_{x\in S}a_{x}(\Ad(E))=(1-g_{X})\dim\dG-\dim \cohog{0}{U,\Ad(E)}.
\end{equation}
Here $a_{x}(\Ad(E))$ is the Artin conductor of $\Ad(E)$ at $x$, see \S\ref{sss:loc mon}, and $g_{X}$ is the genus of $X$.

From \eqref{num rigid} we see that cohomologically rigid $\dG$-local systems exist only when $g_{X}\leq1$. When $g_{X}=1$, rigid examples are very limited (see \cite[\S1.4]{Katz}). 

Most examples of rigid local systems are over open subsets of $X=\PP^{1}$. In this case, if we further assume $\cohog{0}{U,\Ad(E)}=0$, then 
\begin{equation*}
\frac{1}{2}\sum_{x\in S}a_{x}(\Ad(E))=\dim\dG.
\end{equation*}
Compare with \eqref{dim condition}, it is natural to expect that if $E$ comes as the Hecke eigen local system of a weakly rigid \gaud\ $(\bK_{x}, \cK_{x})$, then
\begin{equation*}
\frac{1}{2}a_{x}(\Ad(E))\stackrel{?}{=}[\frg(\cO_{x}): \frk_{x}], \quad \forall x\in S.
\end{equation*}

When $\dG=\SL_{n}$ and the local system is irreducible, the two notions of rigidity are equivalent.
\begin{theorem} Let $E$ be an irreducible rank $n$ $\Qlbar$-local system on $U=X-S$ with trivial determinant, viewed as an $\SL_{n}$-local system.
\begin{enumerate}
\item (Katz \cite[Theorem 5.0.2]{Katz}) If $E$ is cohomological rigid, then it is physical rigid.
\item (L.Fu \cite[Theorem 0.1]{Fu}) If $E$ is  physically rigid, then it is cohomological rigid.
\end{enumerate}
\end{theorem}

\subsection{Applications of rigid \aud} We shall give three applications of the rigid objects in the Langlands correspondence.

\subsubsection{Local systems with exceptional monodromy groups} Katz \cite{Katz-DE} has constructed an example of a local system over $\PP^{1}_{\FF_{p}}-\{0,\infty\}$ whose geometric monodromy lies in the exceptional group $G_{2}$ and is Zariski dense there. This $G_{2}$-local system comes from a rank $7$ rigid local system which is an example of Katz's hypergeometric sheaves. 

The work \cite{HNY}, inspired by the work of Gross \cite{Gross-Prescribe} and Frenkel--Gross \cite{FG},  gives the first examples of local systems (coming from geometry) with Zariski dense monodromy in other exceptional groups $F_{4}, E_{7}$ and $E_{8}$ in a uniform way. The Zariski closure of the geometric monodromy of $\Kl_{\dG}(1,\ph)$ is a connected simple subgroup of $\dG$ of types given by the following table (assuming $p\ne 2,3$)
\begin{center}
\begin{tabular}{|l|l|}
\hline
$\dG$ & \textup{geometric monodromy} \\ \hline
$A_{2n}$ & $A_{2n}$ \\
$A_{2n-1}, C_n$ &  $C_n$ \\ 
$B_n, D_{n+1}$ $(n\geq4)$ & $B_n$   \\
$E_7$ & $E_7$  \\
$E_8$ & $E_8$  \\
$E_6, F_4$ & $F_4$  \\
$B_3,D_4, G_2$ & $G_2$  \\
\hline
\end{tabular}
\end{center}

%

\subsubsection{Motives over number fields with exceptional motivic Galois groups}\label{ss:Serre Q}
In early 1990s, Serre asked the following question \cite{Serre-motive}: Is there a motive over a number field whose motivic Galois group is of exceptional type such as $G_{2}$ or $E_{8}$?

A motive $M$ over a number field $K$ is, roughly speaking, part of the cohomology $\upH^{i}(X)$ for some (smooth projective) algebraic variety $X$ over $K$ and some integer $i$, which is cut out by geometric operations (such as group actions). For each prime $\ell$, the motive $M$ has the associated $\ell$-adic cohomology $\upH_{\ell}(M)\subset \upH^{i}(X_{\overline{K}},\Qlbar)$, which admits a Galois action:
\begin{equation*}
\rho_{M,\ell}:\Gal(\overline{K}/K)\to \GL(\upH_{\ell}(M))
\end{equation*}
The {\em $\ell$-adic motivic Galois group } $G_{M,\ell}$ of $M$ is the Zariski closure of the image of $\rho_{M,\ell}$. This is an algebraic group over $\Qlbar$. Since the motivic Galois groups that appear in the original question of Serre are only well-defined assuming the standard conjectures in algebraic geometry, we shall use the $\ell$-adic motivic Galois group as a working substitute for the actual motivic Galois group (conjecturally they should be isomorphic to each other). Classical groups appear as motivic Galois groups of abelian varieties.  This is why Serre raised the question for exceptional groups only. 

The $G_{2}$ case was answered affirmatively by Dettweiler and Reiter \cite{DR}.

In \cite{Y-motive}, we construct motivic local systems on $\PP^{1}_{\QQ}-\{0,1,\infty\}$ with Zariski dense monodromy in exceptional groups $E_{7},E_{8}$ and $G_{2}$ in a uniform way. They arise as eigen $\dG$-local systems from the \gaud  \  in Example \ref{ex:mot}. As a consequence of this construction, we give an affirmative answer to the $\ell$-adic version of Serre's question for $E_{7}, E_{8}$ and $G_{2}$. With a bit more work, one can also realize $F_{4}$ as a motivic Galois group over $\QQ$ (unpublished). 

Recently, $E_{6}$ has also been realized as a motivic Galois group by Boxer--Calegari--Emerton--Levin--Madapusi-Pera--Patrikis \cite{E6}.

We remark that in the case of $E_{8}$,  no connections are known between $E_{8}$-motives and Shimura varieties. The example given by the rigidity method in \cite{Y-motive} seems to be the only approach, even if one assumes knowledge about cohomology of Shimura varieties. 


\subsubsection{Inverse Galois Problem}
The inverse Galois problem over $\QQ$ asks whether every finite group can be realized as the Galois group of some Galois extension $K/\QQ$. The problem is still open for many finite simple groups, especially those of Lie type.   

The same rigid local systems over $\PP^{1}_{\QQ}-\{0,1,\infty\}$ constructed to answer Serre's question can be used to solve new cases of the inverse Galois problem. We show in \cite{Y-motive} that for sufficiently large primes $\ell$, the finite simple groups $G_{2}(\FF_\ell)$ and $E_{8}(\FF_\ell)$ can be realized as Galois groups over $\QQ$. With a bit more work, one can show that $F_{4}(\FF_\ell)$ is  also a Galois group over $\QQ$. 

In inverse Galois theory, people use the ``rigidity method'' to prove certain finite groups $H$ are Galois groups over $\QQ$. This will be reviewed in the next subsection. In particular, the case of $G_{2}(\FF_{\ell})$ for all primes $\ell\geq5$ was known by the work of Thompson \cite{Thompson} and Feit and Fong \cite{FF}. However,  the case of $E_{8}(\FF_{\ell})$ was known previously only for primes $\ell$ satisfying a certain congruence condition modulo 31 (see the book of Malle and Matzat \cite[II.10]{MM} for a summary of what was known before). 

The $E_{8}$-local system contructed in \cite{Y-motive} also sheds some light to the rigidity method in inverse Galois theory. In fact, in \cite[Conjecture 5.16]{Y-motive} we suggest a rigid triple for $E_{8}(\FF_{\ell})$, which was subsequently proved by Guralnick and Malle \cite{GM}. Their result shows that $E_{8}(\FF_{\ell})$ is a Galois group over $\QQ$ for all primes $\ell>7$.

\subsection{Rigidity in the inverse Galois theory}
It is instructive to compare the notion of rigidity for local systems with the notion of a rigid tuple in inverse Galois theory. We give a quick review following \cite[Chapter 8]{Serre-Galois}. Let $H$ be a {\em finite} group with trivial center. 

\begin{defn} A tuple of conjugacy classes $(C_{1}, C_{2}, \cdots, C_{n})$ in $H$ is called {\em (strictly) rigid}, if
\begin{itemize}
\item The equation 
\begin{equation}\label{ggg}
g_{1}g_{2}\cdots g_{n}=1
\end{equation}
has a solution with $g_{i}\in C_{i}$, and the solution is unique up to simultaneous $H$-conjugacy;
\item For any solutions $(g_{1},\cdots, g_{n})$ of \eqref{ggg}, $\{g_{i}\}_{i=1,\cdots, n}$ generate $H$.
\end{itemize}
\end{defn}

The connection between rigid tuples and local systems is given by the following theorem. Let $S=\{P_{1},\cdots, P_{n}\}\subset\PP^{1}(\QQ)$, and let $U=\PP^{1}_{\QQ}-S$.

\begin{theorem}[Belyi, Fried, Matzat, Shih, and Thompson] Let $(C_{1},\cdots, C_{n})$ be a rigid tuple in $H$.   Then up to isomorphism there is a unique connected unramified Galois $H$-cover $\pi: Y\to U\otimes_{\QQ}{\Qbar}$ such that a topological generator of the (tame) inertia group at $P_{i}$ acts on $Y$ as an element in $C_{i}$.

Furthermore, if each $C_{i}$ is rational (i.e., $C_{i}$ takes rational values for all irreducible characters of $H$), then the $H$-cover $Y\to U\otimes_{\QQ}{\Qbar}$ is defined over $\QQ$. 
\end{theorem}

From the above theorem we see that the notion of a rigid tuple is an analog of physical rigidity for $H$-local systems when the algebraic group $H$ is a finite group.

Rigid tuples combined with the Hilbert irreducibility theorem solves the inverse Galois problem for $H$.
\begin{cor} Suppose there exists a rational rigid tuple in $H$, then $H$ can be realized as $\Gal(K/\QQ)$ for some Galois number field $K/\QQ$.
\end{cor}
For a comprehensive survey of finite simple groups that are realized as  Galois groups over $\QQ$ using rigidity tuples, we refer the readers to the book \cite{MM} by Malle and Matzat.

\subsection{Further directions}

\sss{De Rham point of view}
There is a well-known analogy between $\ell$-adic local systems and connections (i.e., de Rham version of local systems) on algebraic curves. The rigidity method  can be adapted to the de Rham setting by working with $D$-modules on various moduli stacks over $\CC$, and one gets $\dG$-connections on $X-S$ as eigen-connections from Hecke eigen $D$-modules.

In \cite{FG}, the connection version of the Kloosterman local system was first constructed.  In \cite{Zhu}, it is shown that the $D$-module analogue of Hecke eigensheaf construction in \cite{HNY} indeed yields the Frenkel-Gross connection.  There are further works \cite{Ch} and \cite{XZ} in this direction.

\sss{Checking rigidity}
Typically, checking an \aud\ is weakly rigid involves going over certain double cosets (indexed by the affine Weyl group or its cosets) and computing $\Aut(\cE)$ for many $\cE$. In \cite{JY-mic}, we will give a simple criterion for weak rigidity inspired by the microlocal geometry of $\Bun_{G}(\bK_{S})$. Using this criterion, one can check that most of the euphotic automorphic data are weakly rigid.

\sss{Classification of rigid \aud}
Katz \cite{Katz} gave an algorithmic classification of tamely ramified rigid rank $n$ local systems on a punctured $\PP^{1}$ for all $n$. Arinkin \cite{Arin} extended Katz's algorithm to rigid connections with irregular singularities (de Rham version of wildly ramified rigid local systems).  It would be desirable to give a classification of rigid \aud\ for $G=\PGL_{n}$.

\sss{Computing monodromy}
Give a rigid \aud\, one can make predictions on the local monodromy of its Langlands parameter following the principles in \S\ref{sss:matching loc mono}. Verifying these predictions is only done in a small number of cases such as the Kloosterman \auda\ \cite{HNY}, and the tame ramification of the epipelagic \auda\ \cite{Y-Epi}. For example, proving the prediction on the slopes of the epipelagic \auda is an open question.

One the other hand, the global monodromy of the Langlands parameters (Zariski closure of the image of $\pi^{\geom}_{1}(X- S)$ in $\dG$) is unknown for many rigid \auda.


\end{document}